\documentclass[12pt]{article}%
\usepackage{amssymb,amsfonts,amsmath}
\usepackage[nohead]{geometry}
\usepackage[singlespacing]{setspace}
\usepackage[round]{natbib}
\usepackage{color}
\usepackage{graphicx}%
\usepackage{subcaption}
\graphicspath{{figures/}}
\newcounter{condition}
\newtheorem{theorem}{Theorem}[section]
\newtheorem{condition}{Condition}[section]
\newtheorem{corollary}[theorem]{Corollary}

\newtheorem{lemma}[theorem]{Lemma}
\newtheorem{proposition}[theorem]{Proposition}
\newtheorem{remark}[theorem]{Remark}
\newenvironment{proof}[1][Proof]{\noindent\textbf{#1.} }{\ \rule{0.5em}{0.5em}}

\numberwithin{equation}{section}
\numberwithin{theorem}{section}
\numberwithin{condition}{section}

\newcommand{\Cov}{\operatorname{Cov}}

\renewcommand{\P}{\mathbb{P}}
\newcommand{\E}{\mathbb{E}}
\renewcommand{\mathbf}{\boldsymbol}

\newcommand{\abs}[1]{\left\vert#1\right\vert}
\newcommand{\set}[1]{\left\{#1\right\}}
\newcommand{\suit}[1]{\left(#1\right)}

\newcommand{\eps}{\varepsilon}
\renewcommand{\ll}[1]{\mathbf{1}{\set{#1}}}

\makeatletter
\def\@biblabel#1{\hspace*{-\labelsep}}
\makeatother
\begin{document}

\title{\bf Trends in tail dependence of heteroscedastic extremes}

\author{John H.J.\ Einmahl\thanks{\textit{E-mail address:} j.h.j.einmahl@uvt.nl} \vspace{-0.6 cm}\\\textit{Tilburg University}
\vspace{-0.4 cm}\and Chen Zhou\thanks{Postal address: Econometric Institute, Erasmus University Rotterdam, 3000DR Rotterdam, The Netherlands. \textit{E-mail address:} zhou@ese.eur.nl} \vspace{-0.6 cm}\\\textit{Erasmus University Rotterdam}}

\maketitle
\thispagestyle{empty}
\doublespacing
\begin{abstract}
\noindent \singlespacing
\vspace{-0.4 cm}
We consider multivariate extreme value statistics for independent but
nonidentically distributed random vectors. In particular, the data may
have varying tail copulas and also heteroscedastic marginal
distributions. Assuming smoothly changing tail copulas, we propose a
nonparametric estimator for the integrated tail copula
and establish its asymptotic behavior.
Notably, the heteroscedastic marginals do not affect the
limiting processes. We use the main result for the integrated
tail copula to test for a constant tail copula across all observations. Finally, a simulation study shows the good finite-sample behavior of our limit theorems as well as high power of the test.
\vspace{3mm}
\\{\textbf{MSC 2020 subject classifications:} Primary 62G32, 62G20; secondary 60F17.}
\vspace{3mm}
\\{\textbf{Keywords:} Functional central limit theorem, Multivariate extremes, Nonidentical distributions, Tail copula, Tail empirical processes.}
\end{abstract}

\sloppy
\newpage
\section{Introduction}

In bivariate extreme value theory the key object is the tail copula which quantifies the important notion of tail dependence.  Estimation of this tail copula is of prime importance for tail inference.  For a comprehensive treatment of multivariate extreme value statistics, see the monographs  \cite{BGST} and \cite{dHF}.  For statistical analyses along these lines it  is typically assumed that the observations are identically distributed, which is too restrictive for various applications.   Therefore  in the present paper we allow the data to come from different distributions and in particular to have different tail copulas as well as heteroscedastic (see below) marginal distributions.  Assuming that the tail copula is  smoothly  changing  over ``time'', we  propose a nonparametric estimator for the integrated tail copula  as well as one for the tail copula at a  fixed time. Note that such a time varying tail copula is known to appear in various fields of application, in particular in financial markets, oil markets, and climate systems, see, e.g.,  \cite{HH}.

To be more precise, let us specify our setting. Consider independent observations $\set{(X_i,Y_i)}_{i=1}^n$, where the continuous marginal distribution functions of $(X_i,Y_i)$ are $F_i$ and $G_i$, respectively. We model the marginal distributions by the heteroscedastic extremes setup as in \cite{EdHZ}. Assume that there exist two continuous  distribution functions $F_0$ and $G_0$ with endpoints $x^*:=\sup\{x:F_0(x)<1\}$ and $y^*:=\sup\{y:G_0(y)<1\}$, two positive,  continuous functions $c_X$ and $c_Y$, defined on the interval  $[0,1]$,  such that 
\begin{equation*}
   \lim_{x\to x^*}\frac{1-F_i(x)}{1-F_0(x)}=c_X\suit{\frac{i}{n}} \text{\ \ \ and \ \ \ }\lim_{y\to y^*}\frac{1-G_i(x)}{1-G_0(x)}=c_Y\suit{\frac{i}{n}},
   \end{equation*}
hold uniformly for $1\leq i\leq n$ and $n\geq 1$.

We denote the distribution function of  $(1-F_i(X_i),1-G_i(Y_i))$ with $C_i$, a (survival) copula.  Each $C_i$ defines a tail copula, while there is smooth variation of these tail copulas across $1\leq i\leq n$. This is formalized as follows. 
Assume that there exists a function $R(u,v;s)$, defined on $[0,\infty )^2\times[0,1]$, which is a tail copula for each $s\in [0,1]$, such that for each  $T>0$,
\begin{equation} \label{eq:definition of R}
    \lim_{t\to 0}\frac{1}{t}C_i(tu,tv)=R\suit{u,v;\frac{i}{n}},    
\end{equation}
holds uniformly for $(u,v)\in[0,T]^2$, $1\leq i\leq n$ and $n\geq 1$. 

 First we estimate the function $R(u,v;s)$ locally for fixed $s$.  
 Next define  the integrated tail copula $$I_R(u,v;s)=\int_0^s R(u,v;w)dw.$$ Observe that $I_R(u,v;1)$ denotes the average tail copula,  representing the average tail dependence.   For general $s\in [0,1] $, we estimate $I_R(u,v;s)$ by adding up local estimators. 
 The latter estimator is particularly important in testing trends in  tail dependence. The precise definitions of both estimators are given in the next section.

The main goal of this paper is to establish the novel theory on the asymptotic behavior of these estimators, revealing in particular the non-trivial limiting process of the standardized estimator of the integrated tail copula. Obviously, the present flexible model and the asymptotic behavior of the estimators pave the way for  extreme value statistics in the very relevant situation that the tail dependence structure is time varying. Interestingly, the heteroscedastically varying marginals do not show up in the limiting processes and hence need not be estimated for tail inference.   We use the powerful result  for the integrated tail copula estimator in particular to test if there is indeed a time changing tail copula or that it remains constant over time. The proof of the main limit theorem was very challenging due to the variation in tail dependence and marginal distributions, as well as the diverging number of local estimators used in the integrated tail copula estimator. Handling these challenges  requires   various innovative steps.

In univariate extreme value theory a smoothly changing extreme value index has been studied rather recently in \cite{dHZ}. The motivation, explanation,  and applications given therein for a time changing extreme value index can directly be translated to a changing tail copula, as in the present paper.  

For multivariate tail dependence  related but very different work has been published very recently. In \cite{Drees} the setup where  the data are multivariate regularly varying with a time varying spectral measure and tail index is studied. In such a case the marginal extreme value indices (at a fixed time) need to exist, to be positive and to be equal, which is restrictive and avoided with our approach.
It is well-known and easy to see that the tail copula remains the same when the components of the bivariate data are transformed by  increasing transformations, hence tail heaviness should not play a role when studying a time varying tail copula. \cite{EZ} studies a constant (over time) tail copula, with \textit{equal} scedasis functions for the marginals, that is, only the marginals vary over time, not the tail dependence. 
The setup in  \cite{HH} is quite similar to the present one,   but the focus therein is on testing for equality of the marginal scedasis functions. In that paper the scedasis functions are relevant and appear in the limiting Gaussian process, whereas we consider the scedasis functions as nuisance and ``correctly'' obtain a novel limiting process that does not contain these scedasis functions.

The remainder of this paper is organized as follows. In Section 2 we present the estimators, the conditions, and the main result unraveling the asymptotic behavior of the estimator of $I_R$; we then use this result to test for a constant tail copula. The non-standard, intricate proofs are detailed in Section 3. 
Finally,  the finite-sample performance of the estimator and the  test
is examined in Section 4. 


\section{Main Results}

In this section we first present the nonparametric estimator  of $R(u,v;s)$, locally for given $s\in(0,1)$, and then that of
   the integrated tail copula $I_R(u,v;s)$. Next we state natural conditions for the proof of the main result, the asymptotic normality of the estimator of $I_R$. In particular the relative growth rate between the usual intermediate sequence $k$ and the bandwidth $h$ (partitioning $[0,1]$, the range of $s$) is delicate. In the last subsection we present the aforementioned limit theorem revealing the asymptotic behavior of the estimator of $I_R$ and we discuss tests for a constant -- over $s$ -- tail copula $R$. 

\subsection{The estimators}
Consider an intermediate sequence $k=k(n)$, that is, $k/n\to 0 $ and $k\to\infty$ as $n\to\infty$. In addition, let  $h=h(n)$ be a bandwidth such that $h\to 0$, $kh\to\infty$, as $n\to\infty$, and $nh$ is an integer.

Denote $I_{s,n}=\set{1\leq i\leq n: s-\frac{h}{2}<\frac{i}{n}\leq s+\frac{h}{2}}$. Clearly $\abs{I_{s,n}}=nh$ for all $\frac{h}{2}\leq s\leq 1-\frac{h}{2}$. To estimate $R(u,v;s)$, we use the following local estimator
 \begin{equation} \label{eq:estimator for R function: multivariate}
   \widehat R(u,v;s)=\frac{1}{kh}\sum_{i\in I_{s,n}} \ll{X_i> X^{(s)}_{nh-\lfloor khu\rfloor,nh},Y_i>Y^{(s)}_{nh-\lfloor khv\rfloor ,nh} },
 \end{equation}
 where $X^{(s)}_{nh-\lfloor khu\rfloor ,nh}$ is the $(nh-\lfloor khu \rfloor)-$th highest value among the observations $\set{X_i}_{i\in I_{s,n}}$, $Y^{(s)}_{nh-\lfloor khv \lfloor,nh}$ is defined similarly. 
 The definition in (\ref{eq:estimator for R function: multivariate}) holds for $s\in[\frac{h}{2},1-\frac{h}{2}]$ and $(u,v)\in [0,T]^2$ for $T>0$. For any fixed $s\in (0,1)$, $\widehat R(u,v;s_)$ is well-defined  for large $n$.
 %

 We  add up the local estimators to obtain an estimator for  
 $I_R(u,v;s)=\int_0^s R(u,v;w)dw$. Denote $m:= \lfloor1/h \rfloor$, 
 We consider an alternative estimator of $R(u,v;s)$ for $s\in(0,1]$ as 
 $$\widetilde R(u,v;s)=\begin{cases}
     \widehat R(u,v;(j-1/2)h) & \text{ if } s\in ((j-1)h, jh] \text{ for } j=1,2,\ldots, m;\\
     \widehat R(u,v;(m-1/2)h) & \text{ if }  s\in (mh,1].
 \end{cases}$$ 
 This alternative estimator  is a piecewise constant  version of $\widehat R(u,v;\cdot )$,    that on each interval $((j-1)h, jh]$,  takes the value at its midpoint. Based on $\widetilde R$, we define the estimator for $I_R$ as follows:
\begin{equation} \label{eq:estimator for tilde R function: multivariate}
\widehat I_R(u,v;s)=\int_0^s \widetilde R(u,v;w) dw.
\end{equation}

\subsection{Conditions}
The first condition controls the speed of convergence in the marginal heteroscedastic models. It is a standard compatibility condition assuming away the asymptotic estimation bias,  see e.g. \cite{EdHZ}.
\begin{condition} \label{eq:condition heteroskedastic extremes}
   There exists two eventually non-increasing positive functions $A_X$ and $A_Y$ with $\lim_{t\to\infty}A_X(t)=\lim_{t\to\infty}A_Y(t)=0$, such that as $x\to x^*$ and $y\to y^*$
   \begin{align*}
   &\sup_{n\in\mathbb{N}}\max_{1\leq i\leq n}\abs{\frac{1-F_i(x)}{1-F_0(x)}-c_X\suit{\frac{i}{n}}}=O\suit{A_X\suit{\frac{1}{1-F_0(x)}}},\\
   &\sup_{n\in\mathbb{N}}\max_{1\leq i\leq n}\abs{\frac{1-G_i(y)}{1-G_0(y)}-c_Y\suit{\frac{i}{n}}}=O\suit{A_Y\suit{\frac{1}{1-G_0(y)}}}.
   \end{align*}
\end{condition}
Since $c_X$ is a positive continuous function on $[0,1]$, there exists $0<\underline{c}_X<\overline{c}_X<\infty$ such that $\underline{c}_X\leq c_X(s)\leq \overline{c}_X$ for all $s\in [0,1]$. We further define $\underline{c}_Y$ and $\overline{c}_Y$ in a similar way.

The following condition balances the choices of the intermediate sequence $k$ and the bandwidth $h$. Such conditions are common in the smoothing literature.
\begin{condition} \label{eq:condition on k} As $n\to\infty$,
   $$k\to\infty, \, k/n\to 0,\,  h\to 0,\,  kh^3/\log^3 n\to \infty\text{\ and \ } kh^4\to 0.$$
\end{condition}

The next condition assumes the sequence $k$ is such that the convergences  to the scedasis functions $c_X, c_Y$ and to $R$, respectively, are fast enough to obtain centered limit distributions. 
\begin{condition} \label{eq:condition k extra}
   For all $\tau>0$, the sequence $k:=k(n)$  satisfies, as $n\to\infty$,
   \begin{align}
       &\sqrt{k}A_X\suit{\frac{n}{k\tau }}\to 0 \text{\ \ and\ }\sqrt{k}A_Y\suit{\frac{n}{k\tau }}\to 0,    \label{eq:condition k related to marginals}\\
      &\sqrt{k}\max_{1\leq i\leq n}\sup_{ (u,v)\in[0,\tau]^2}\abs{\frac{n}{k}C_i\suit{\frac{k}{n}u,\frac{k}{n}v}-R\suit{u,v;\frac{i}{n}}}\to 0. \label{eq:condition k related to R}
   \end{align}
   \end{condition}
   
 
Next we assume  some smoothness of the scedasis functions.
\begin{condition} \label{eq:condition smooth of c and R}
The scedasis functions $c_X$ and $c_Y$ are Lipschitz continuous.
   \end{condition}

Lastly, we assume that the partial derivatives of $R(u,v;s)$ with respect to $u$, $v$  and $s$, denoted as $R_{1}(u,v;s)$, $R_2(u,v;s)$, and $R'(u,v;s)$, respectively, are continuous on suitable regions. Further denote $R''(u,v;s):=\frac{\partial}{\partial s} R'(u,v;s)$. 

\begin{condition} \label{eq:condition partial derivatives of R}
   The partial derivative $R_1(u,v;s)$ is continuous on $(0,\infty)\times[0,\infty)\times [0,1]$ and the partial derivative $R_2(u,v;s)$ is continuous on $[0,\infty)\times(0,\infty)\times [0,1]$. Moreover, for each $s$, $R_1(u,v;s)$ and $R_2(u,v;s)$  are Lipschitz-continuous on $[0,\infty)^2\setminus[0, \frac{1}{2})^2$ in $u$ and in $v$, with some constant $K$ not depending on $s$, and, for all $\tau>0$,  $R''(u,v;s)$ is  bounded on $[0,\tau]^2\times [0,1]$.
\end{condition}
In the simulation study, we provide examples of such smooth functions $c_X$, $c_Y$ and $R$.

\subsection{Theorems}

Now we establish the main result, the limit theorem for  $\widehat I_R(u,v;s)$. Let $W_I$ be a centered Gaussian process on $[0,\infty]^2\setminus\{\infty, \infty\}\times  [0,1]$ with covariance structure
\begin{equation}\label{cova}\Cov( W_I(u_1,v_1;s), W_I(u_2,v_2;t))=\int_0^{s\wedge t} R(u_1\wedge u_2, v_1\wedge v_2; w)dw,\end{equation}
 for any $(u_1,v_1,s),(u_2,v_2,t)\in [0,\infty]^2\setminus\{\infty, \infty\}\times [0,1]$.
For the integrated  estimator $I_R$,  we obtain the joint asymptotic behavior at the relatively fast  rate $1/\sqrt{k}$, instead of  $1/\sqrt{kh}$.
   \begin{theorem} \label{theorem:global estimator for multivariate}
      Assume that Conditions \ref{eq:condition heteroskedastic extremes} - \ref{eq:condition partial derivatives of R}  hold. Then for all $T>0$, 
      as $n\to\infty$, \begin{eqnarray*}&&\!\!\!\!\!\!\!\!\!\!\!\!\sqrt{k}\suit{\widehat I_R(u,v;s)-I_R(u,v;s)}
      \\&&\qquad
      \leadsto  W_I(u,v;s)-\int_{0}^s R_1(u,v,w)W_I(u, \infty; dw)-\int_{0}^s R_2(u,v;w)W_I( \infty, v;dw), \end{eqnarray*}
    where $\leadsto$ denotes weak convergence on $D_{[0,T]^2\times[0,1]}$  equipped with the Skorokhod distance.
   \end{theorem}
Note that the limit process is non-degenerate  if $R(1,1;s)>0$, for some $s \in [0,1]$.

There are two interesting special cases resulting from this theorem.
When taking $u=v=1$, $R(1,1;s)$ is the so-called tail dependence coefficient. Correspondingly, the integrated tail dependence coefficient $I_R(s):=I_R(1,1;s)$ can be estimated by $\widehat I_R(s)=\widehat I_R(1,1;s)$. From  Theorem \ref{theorem:global estimator for multivariate}, we immediately obtain that
$$\sqrt{k}\suit{\widehat I_R(s)-I_R(s)}\leadsto W\suit{\int_0^s \sigma^2(w)dw},$$
where $\sigma^2(w)=(1-R(1,1;w))(R(1,1;w)-2R_1(1,1;w)R_2(1,1;w))$ and $W$ is a standard univariate Wiener process. On the other hand, when taking $s=1$, we obtain  $\widehat I_R(u,v;1)$, an estimator of the average tail copula across all observations. 

   To test the null hypothesis that there is no trend in $R$, i.e., $R(u,v;s)\equiv R(u,v)$ for all $(u,v) \in [0,T]^2, \, s\in[0,1]$, we consider the thus obtained special case $I_R(u,v;s)=sR(u,v)$. The limit process in Theorem \ref{theorem:global estimator for multivariate} then reduces to $$W_I(u,v;s)-R_1(u,v)W_I(u,\infty;s)-R_2(u,v)W_I(\infty,v;s),$$ where $R_1(u,v)$ and $R_2(u,v)$ are the partial derivatives of $R(u,v)$ with respect to $u$ and $v$, respectively, and where
 the covariance structure
   in (\ref{cova}) now specializes to $$\Cov( W_I(u_1,v_1;s), W_I(u_2,v_2;t))=R(u_1\wedge u_2, v_1\wedge v_2)(s\wedge t) ,
   $$
   that is, $W_I$ is a Wiener process with as trivariate ``time''  the $R$-measure times Lebesgue measure.   Define on  $([0,\infty]^2\setminus\{\infty, \infty\}\times  [0,1]$,
  $$B_I(u,v;s)=W_I(u,v;s)-sW_I(u,v;1).  $$ Theorem \ref{theorem:global estimator for multivariate} now immediately yields  the following result  for this special case. \begin{corollary}\label{corollary:global estimator no trend}
      Assume that Conditions \ref{eq:condition heteroskedastic extremes} - \ref{eq:condition partial derivatives of R}  hold and that $R(u,v;s)\equiv R(u,v)$ for all $s\in[0,1]$. 
      Then, as $n\to\infty$,
      \begin{multline*}\sqrt{k}\suit{\widehat I_R(u,v;s)-s\widehat I_R(u,v;1)}\\
      \leadsto  B_I(u,v;s)- R_1(u,v)B_I(u,\infty;s)- R_2(u,v)B_I(\infty,v;s)=:B(u,v;s),
      \end{multline*}
    where $\leadsto$ denotes weak convergence on $D_{[0,T]^2\times[0,1]}$  equipped with the Skorokhod distance,
    and hence
    \begin{eqnarray*}
   &&\label{bee}\sup_{(u,v,s)\in [0,T]^2\times[0,1]}
   \left|\sqrt{k}\suit{\widehat I_R(u,v;s)-s\widehat I_R(u,v;1)}\right| \stackrel{d}{\to} \sup_{(u,v,s)\in [0,T]^2\times[0,1]}\left|B(u,v;s)
     \right|, 
   \\ \nonumber\\
&&\label{bee2} \!\!\!\!\!\!\!\int_0^1\int_0^T\int_0^T\!
   \left(\sqrt{k}\suit{\widehat I_R(u,v;s)-s\widehat I_R(u,v;1)}\right)^2\!dudvds\stackrel{d}{\to} \int_0^1\int_0^T\int_0^T\!\!B^2(u,v;s)
     dudvds. 
   \end{eqnarray*}
   
      \end{corollary}

    We need to estimate $R_1$ and $R_2$. We only consider the estimation of $R_1$; the result for $R_2$ follows similarly. We simplify the notation by writing $\widehat R(u,v):= \widehat I_R(u,v;1)$ and define
    $$\widehat R_1(u,v)=\frac{\widehat R(u+k^{-1/5},v)-\widehat R(\max\{u-k^{-1/5},0\},v)}
    {\min \{ 2k^{-1/5},u+k^{-1/5}\}}\, , \quad (u,v) \in [0,T]^2.$$
    From Theorem \ref{theorem:global estimator for multivariate} with $s=1$ (and $T$ replaced with $2T$), we obtain immediately that 
    $$\sup_{(u,v) \in [0, 2T]^2}k^{1/5}|\widehat R(u,v)- R(u,v)|\stackrel{\mathbb{P}}{\to} 0,$$
    which yields
    \begin{equation} \label{conr1}\sup_{(u,v) \in [0, T]^2}|\widehat R_1(u,v)- R_1(u,v)|\stackrel{\mathbb{P}}{\to} 0.\end{equation}

Consider a modification of the Gaussian limiting process in Corollary \ref{corollary:global estimator no trend}, where $R, R_1, R_2$ are replaced
by $\widehat R, \widehat R_1,\widehat R_2$, \textit{conditional} on $\widehat R$, and denote it with $\widehat B$. In other words, we compute the estimates (the realizations) of 
$\widehat R$, and then $ \widehat R_1,\widehat R_2$,  from the data and subsequently generate the Gaussian process $\widehat B$. 
Let $Q_{n,1}$ be the quantile function of the random variable $\sup_{(u,v,s)\in [0,T]^2\times[0,1]} |\widehat B(u,v;s)|$, conditional on $\widehat R$, and  let $Q_1$ be the  quantile function of the limiting random variable with $\widehat B$ replaced by $B$; similarly let $Q_{n,2}$ be the quantile function of the random variable $\int_0^1\int_0^T\int_0^T\widehat B^2(u,v;s)dudvds$, conditional on $\widehat R$, and  let $Q_2$ be the  quantile function of the limiting random variable with $\widehat B$ replaced by $B$.
 \begin{corollary} \label{24} Let $\alpha \in (0,1)$. Under the conditions of Corollary \ref {corollary:global estimator no trend}, we have for   each continuity point  $1-\alpha$  of  $Q_j$, that 
 $$  Q_{n, j}(1-\alpha)\stackrel{\mathbb{P}}{\to} Q_j(1-\alpha), \quad \mbox{as } n\to \infty, \quad  j=1,2.$$\end{corollary}
 This shows that the simulation of $\widehat B$ yields asymptotically correct critical values for testing that   $R(u,v;s)$ has no trend in $s$.

The proof of  Corollary \ref{24} is somewhat lengthy, but routine. A similar but much more difficult result, with a detailed proof, is presented in \cite{EdHL},  Corollary 4.3.
Here we omit the proof and only mention that it uses, in particular, (\ref{conr1}) and the fact that
for $u\geq 0$, $\widehat R(u, \infty)=\widehat R( \infty, u)= \lfloor khu \rfloor/(kh)$, almost surely.

\setcounter{theorem}{0}
\begin{remark} For completeness we here state the asymptotic behavior of the local estimator $\widehat R(u,v;s)$, see (\ref{eq:estimator for R function: multivariate}),  at a fixed point $s\in(0,1)$ and  $(u,v)\in[0,T]^2$.   The result holds under the present conditions, but weaker conditions would suffice. Let $W$  be a centered Gaussian process  defined on $[0,\infty]^2\setminus\{\infty, \infty\}$, with covariance structure
 $$\Cov(W(u_1,v_1), W(u_2,v_2))=R(u_1\wedge u_2,v_1\wedge v_2;s),$$
 for any $(u_1,v_1),(u_2,v_2)\in [0,\infty]^2\setminus\{\infty, \infty\}$.
Then we have $$  \sqrt{kh}(\widehat R(u,v;s)-R(u,v;s))\leadsto W(u,v)-R_1(u,v;s) W(u,\infty)-R_2(u,v;s) W(\infty,v),$$
    where $\leadsto$ denotes weak convergence on $D_{[0,T]^2}$  equipped with the Skorokhod distance. 
    Note that the data are `almost' identically distributed for indices $i$ in the considered  small neighborhood of $s$. Actually, when we omit $s$ from the notation  and replace $kh$ with $k$ we obtain the result in i.i.d.\ case. 
\end{remark}

\begin{remark} Theorem \ref{theorem:global estimator for multivariate}, our main result, could be  generalized to data in dimension higher than two. We do not pursue this, but discuss it here briefly.
Note that in dimension higher than two,   the tail dependence function $\ell$ is better suited
for describing tail dependence than $R$, see, e.g.,  \cite{EKS} for its definition and for asymptotical normality of the estimator of $\ell$ in higher dimensions.  However, in dimension two, $\ell$ and $R$  provide the same information, since
$\ell(u,v;s) = u+v - R(u,v;s)$. Hence an estimator of
$I_\ell(u,v;s):=\int_0^s \ell(u,v;w)dw$ is $\widehat I_\ell(u,v;s):=s(u+v)-\widehat I_R(u,v;s)$. We immediately obtain the asymptotic normality of this estimator from Theorem \ref{theorem:global estimator for multivariate}, but we could also easily prove it directly by appropriately modifying the present proof of Theorem \ref{theorem:global estimator for multivariate}. This modified version of the proof for   $\widehat I_\ell$ could be routinely adapted to arbitrary dimension, but such a proof would be notationally more unwieldy. 
\end{remark}

\section{Proofs}

\subsection{Notations}
Recall that $m= \lfloor 1/h \rfloor$. For each $1\leq j\leq m$, denote the center of the interval $((j-1)h, jh]$ as $s_j=jh-\frac{h}{2}$. The corresponding index set of observations falling into this interval is defined as
 $I_{j}=\set{1\leq i\leq n: s_j-\frac{h}{2}< \frac{i}{n} \leq s_j+\frac{h}{2}}$. In addition, for each $s\in(0,1]$, denote 
 $j(s)=\lceil s/h\rceil \wedge m$. 
 We have that  $\widetilde R(u,v;s) = \widehat R(u,v; s_{j(s)}) $, for all $s\in(0,1].$ 

Further, denote  $H_{F,i}=\suit{\frac{1-F_i}{c_X(i/n)}}^{\leftarrow}$, where $\cdot^{\leftarrow}$ defines the left-continuous inverse. Similarly we define $H_{G, i}$ for the $Y$-coordinate. Then, write $(X_i,Y_i):=(H_{F,i}(U_i), H_{G,i}(V_i)), i=1, \ldots, n$, where $U_i$ is uniformly distributed on  $(0, 1/c_X(i/n))$, $V_i$ is uniformly distributed on  $(0, 1/c_Y(i/n))$, and $(c_X(i/n)U_i, c_Y(i/n)V_i)$ follows the copula $C_i$, $i=1, \ldots,  n$. In addition, $(U_i,V_i), i=1, \ldots, n,$ are independent random vectors.

Next, in each interval $I_j$, define,$$\underline{H}_{F,j}(x)=\min_{i\in I_j} H_{F,i}(x),\text{\ \ and \ }\overline{H}_{F,j}(x)=\max_{i\in I_j} H_{F,i}(x).$$
Further, define the order statistics of $\set{U_i}_{i\in I_j}$ as $U^{(j)}_{1,nh}\leq U^{(j)}_{2,nh} \leq \ldots \leq  U^{(j)}_{nh,nh}$.
Since both  $\underline{H}_{F,j}$ and $\overline{H}_{F,j}$  are non-increasing, we have 
$$\underline{H}_{F,j}\suit{U^{(j)}_{\lfloor khu \rfloor +1,nh}}\leq X^{(s_j)}_{nh- \lfloor khu \rfloor,nh}\leq \overline{H}_{F,j}\suit{U^{(j)}_{ \lfloor khu \rfloor +1,nh}}.$$
Similar inequalities hold for the $Y$-coordinate, with the obvious notations.

Define $\widehat{R}_L(u,v;s_j)$ and $\widehat{R}_U(u,v;s_j)$ as follows:
   \begin{eqnarray*}\nonumber\label{eq:lowerbound for R general s}
     &&\!\!\!\!\!\!\!\!\!\widehat{R}_{L}(u,v;s_j)= \frac{1}{kh}\sum_{i\in I_j}
     \mathbf{1}\!\left\{U_i< \frac{1-F_i\suit{\overline{H}_{F,j}\suit{ U^{(j)}_{\lfloor khu \rfloor+1,nh}}}}{c_X(i/n)},V_i<\frac{1-G_i\suit{\overline{H}_{G,j}\suit{V^{(j)}_{\lfloor khv \rfloor+1,nh}}}}{c_Y(i/n)}\right\}
     \\ \label{eq:upperbound for R general s} \nonumber
      &&\!\!\!\!\!\!\!\!\!\widehat{R}_{U}(u,v;s_j)= \frac{1}{kh}\sum_{i\in I_j}
     \mathbf{1}\!\left\{U_i< \frac{1-F_i\suit{\underline{H}_{F,j}\suit{U^{(j)}_{\lfloor khu \rfloor+1,nh}}}}{c_X(i/n)},V_i<\frac{1-G_i\suit{\underline{H}_{G,j}\suit{V^{(j)}_{\lfloor khv \rfloor+1,nh}}}}{c_Y(i/n)}\right\}\!.
   \end{eqnarray*} 
Then we have the following inequalities
\begin{equation} \label{eq:inequality bounding R for general s}
      \widehat{R}_L(u,v;s_j)\leq \widehat{R}(u,v;s_j)\leq \widehat{R}_U(u,v;s_j),
\end{equation} 
and hence
   \begin{equation} \label{eq:inequality bounding I_R, general s}
    \widehat{I}_{R, L}(u,v;s)
    \leq\widehat{I}_R(u,v;s)\leq \   \widehat{I}_{R, U}(u,v;s),
  \end{equation}
  where  $\widehat{I}_{R, L}$ is defined similarly as  $\widehat{I}_R$, with $\widehat{R}$ replaced by  $\widehat{R}_L$, and  $\widehat{I}_{R, U}$ is obtained from $\widehat{I}_R$ by replacing  $\widehat{R}$ by  $\widehat R_U$.

The goal is to show that both the upper and lower bound share the same limit as stated in Theorem \ref{theorem:global estimator for multivariate}. We provide the proof for the upper bound only; the proof for the lower bound follows similar lines.
\subsection{Two preliminary lemmas}

 The following lemma  shows that one can replace $\frac{1-F_i}{c_X(i/n)}\suit{\underline{H}_{F,j}\suit{\cdot}}$ by the identity function uniformly for all $i\in I_{j}$ and for all $j=1,2,\ldots,m$. 
\begin{lemma}\label{lemma: delta_F, general s}
   Denote
   $$\Delta_F:=\max_{1\leq j\leq m}\sup_{u\in[0,T]}\sup_{i\in I_j}\abs{\frac{1-F_i\suit{\underline{H}_{F,j}\suit{\frac{ku}{n}}}}{c_X\suit{\frac{i}{n}}\frac{ku}{n}}-1}.$$
   Under the conditions in Theorem \ref{theorem:global estimator for multivariate}, we have that for any $\eps>0$, 
   $$\lim_{n\to\infty}\sqrt{k}\Delta_F=0.$$
  The same  limit relation holds for $\Delta_G$ defined in a similar way for the second coordinate.
\end{lemma}

The proofs of this and the next lemma are postponed to  the final subsection.

We intend  to apply Lemma \ref{lemma: delta_F, general s} with $u$ replaced by $\frac{n}{k}U^{(j)}_{\lfloor khu \rfloor+1,nh}$, for $j=1,2,\ldots,m$, which could be somewhat larger than $T$. To quantify this we will use the following result. 
\begin{lemma}\label{lemma: asymptotics for the quantiles}
   Under the conditions in Theorem \ref{theorem:global estimator for multivariate}, we have that as $n\to\infty$, \begin{align*}
     &\max_{1\leq j\leq m}\sup_{u\in[0,T]}\abs{b_{X,j}\frac{n}{k}U^{(j)}_{\lfloor khu \rfloor+1,nh}-u}=O_P\suit{\sqrt{\frac{m\log n}{k}}},\\
     &\max_{1\leq j\leq m}\sup_{v\in[0,T]}\abs{b_{Y,j}\frac{n}{k}V^{(j)}_{\lfloor khv \rfloor+1,nh}-v}=O_P\suit{\sqrt{\frac{m\log n}{k}}},
   \end{align*}
   where $b_{X,j}:=\frac{1}{nh}\sum_{i\in I_j}c_X(i/n)$ and $b_{Y,j}$ is defined analogously. 
   In addition,  for $0\leq b\leq 1/3$ we have
   \begin{align*}
     &\max_{1\leq j\leq m}\sup_{u\in\left[0,\frac{1}{k^b}\right]}\abs{b_{X,j}\frac{n}{k}U^{(j)}_{\lfloor khu \rfloor+1,nh}-u}=O_P\suit{\sqrt{\frac{m\log n}{k^{1+b}}}},\\
     &\max_{1\leq j\leq m}\sup_{v\in\left[0,\frac{1}{k^b}\right]}\abs{b_{Y,j}\frac{n}{k}V^{(j)}_{\lfloor khv \rfloor+1,nh}-v}=O_P\suit{\sqrt{\frac{m\log n}{k^{1+b}}}}.
   \end{align*}
\end{lemma}

\subsection{Proof of Theorem \ref{theorem:global estimator for multivariate}}
 Due to symmetry, we only provide the proof for the upper bound $\widehat{I}_{R, U}(u,v;s)$. Let $T'>T$. Lemma \ref{lemma: asymptotics for the quantiles} yields that with probability tending to 1, $\frac{n}{k} U^{(j)}_{\lfloor khu \rfloor+1,nh}\leq T'$ and $\frac{n}{k}V^{(j)}_{\lfloor khv \rfloor+1,nh}\leq T'$ for all $(u,v)\in[0,T]^2$ and all $j=1,2,\ldots,m$. On the set where these inequalities hold, we can apply Lemma \ref{lemma: delta_F, general s} with $T$ replaced by $T'$, 
to obtain that
\begin{align*}
   &(1-\Delta_F) U^{(j)}_{\lfloor khu \rfloor+1,nh} \leq \frac{1-F_i\suit{\underline{H}_{F,j}\suit{U^{(j)}_{\lfloor khu \rfloor+1,nh}}}}{c_X\suit{\frac{i}{n}}}\leq (1+\Delta_F)  U^{(j)}_{\lfloor khu \rfloor+1,nh},\\
   &(1-\Delta_G) V^{(j)}_{\lfloor khv \rfloor+1,nh} \leq \frac{1-G_i\suit{\underline{H}_{G,j}\suit{V^{(j)}_{\lfloor khv \rfloor+1,nh}}}}{c_Y\suit{\frac{i}{n}}}\leq (1+\Delta_G)  V^{(j)}_{\lfloor khv \rfloor+1,nh}. 
\end{align*}
Hence, we have that
$$\widehat{I}_{R, UL}(u,v;s)\leq \widehat{I}_{R, U}(u,v;s) \leq \widehat{I}_{R, UU}(u,v;s),$$
where
$  \widehat I_{R,UU}(u,v;s)=\int_0^s \widetilde R_{UU}(u,v;w)dw,$ 
and
    \begin{align*}
    \widetilde R_{UU}(u,v;w)&=\begin{cases}
        \widehat R_{UU}(u,v;s_j) &\text{ if } w\in ((j-1)h, jh], \text{  for } j=1,2,\ldots,m,\\
        \widehat R_{UU}(u,v;s_m) &\text{ if } w> mh,
    \end{cases}\\
    \widehat R_{UU}(u,v;s_j) &= \frac{1}{kh} \sum_{i\in I_j}\ll{U_i< \frac{k}{n}\hat u_j,V_i<\frac{k}{n}\hat v_j},
\end{align*}
with $\hat u_j=\frac{n}{k}U^{(j)}_{\lfloor khu \rfloor+1,nh}\suit{1+\Delta}$, $\hat v_j= \frac{n}{k}V^{(j)}_{\lfloor khv \rfloor+1,nh}\suit{1+\Delta}$ and $\Delta =\max(\Delta_F,\Delta_G)$.
Here the function $\widetilde R_{UU}$ extends the definition of $\widehat R_{UU}$ to the entire interval $s\in(0,1]$ in a piecewise constant way, similar to $\widetilde R(u,v;s)$.
The lower bound  $\widehat{I}_{R, UL}(u,v;s)$ is defined in a similar way except that $\Delta$ is replaced by $-\Delta$. 
Again, we will show that the upper and lower bounds, $\widehat{I}_{R, UU}(u,v;s)$ and 
$\widehat{I}_{R, UL}(u,v;s)$ share the same  limit as stated in Theorem 
\ref{theorem:global estimator for multivariate}. Due to symmetry, we only consider $\widehat{I}_{R, UU}(u,v;s)$.

Define for $(u,v)\in [0,\infty]^2\setminus\{\infty, \infty \}$, for each $1\leq j\leq m$,
$$w_{n,j}(u,v):=\sqrt{kh}\suit{\frac{1}{kh}\sum_{i\in I_j}\left[\ll{U_i< \frac{k}{n}u,V_i<\frac{k}{n}v}-C_i\suit{\frac{c_X(i/n)ku}{n},\frac{c_Y(i/n)kv}{n}}\right]};$$ note that  
$C_i\suit{z,\infty}= C_i\suit{\infty, z}= z, z \in [0,1]$.
We start by considering $s\in\set{jh: 1\leq j\leq m}$. For such $s$ we have
\begin{eqnarray*}
&&\sqrt{k}\suit{\widehat{I}_{R,UU}(u,v;s)- I_R(u,v;s)}\\
   &&=\sqrt{k}\left(
   \frac{1}{k}\sum_{j=1}^{s/h} \sum_{i\in I_j}\ll{ U_i< \frac{k}{n}\hat u_j, V_i<\frac{k}{n}\hat v_j}- \int_0^{s} R(u,v;w)dw\right)\\
  && =\,\sqrt{h}\sum_{j=1}^{s/h}w_{n,j}(\hat u_j,\hat v_j)
   \\
   &&+\sqrt{k}\left(\frac{1}{n}\sum_{j=1}^{s/h}\sum_{i\in I_j}\frac{n}{k}C_i\suit{\frac{c_X(i/n)k}{n}\hat u_j,\frac{c_Y(i/n)k}{n}\hat v_j}\right.
   \\&&\qquad\left.
   -\frac{1}{n}\sum_{j=1}^{s/h}\sum_{i\in I_j}R\suit{c_X\suit{\frac{i}{n}}\hat u_j,c_Y\suit{\frac{i}{n}}\hat v_j;s_j} \right)\\
   &&+\sqrt{k}\left(\frac{1}{n}\sum_{j=1}^{s/h}\sum_{i\in I_j}R\suit{c_X\suit{\frac{i}{n}}\hat u_j,c_Y\suit{\frac{i}{n}}\hat v_j;s_j}
-h\sum_{j=1}^{s/h}
R\suit{b_{X,j}\hat u_j,b_{Y,j}\hat v_j;s_j}\right)\\
&&+\sqrt{k}\left(h\sum_{j=1}^{s/h}
R\suit{b_{X,j}\hat u_j,b_{Y,j}\hat v_j;s_j}
   -h\sum_{j=1}^{s/h}R(u,v;s_j)
   \right)\\
&&+\sqrt{k}\suit{h\sum_{j=1}^{s/h}R(u,v;s_j)
-\int_0^{s} R(u,v;w)dw}\\
  && =:T_1+T_2+T_3+T_4+T_5.
\end{eqnarray*}

We handle the five parts separately in five propositions; their proofs are given in the next subsection.
\begin{proposition} \label{prop: uniform convergence of I1}
      Under the conditions in Theorem \ref{theorem:global estimator for multivariate}, we have that uniformly for all $(u,v,s)\in [0,T]^2\times \set{jh: 1\leq j\leq m}$,    
$$T_1=\sqrt{h}\sum_{j=1}^{s/h}w_{n,j}\suit{\frac{u}{b_{X,j}},\frac{v}{b_{Y,j}}}
+o_P(1)=:\widetilde T_1+o_P(1).
    $$ 
\end{proposition}

\begin{proposition} \label{prop: uniform convergence of I2}
   Under the conditions in Theorem \ref{theorem:global estimator for multivariate}, we have that uniformly for all $(u,v,s)\in [0,T]^2\times \set{jh: 1\leq j\leq m}$, as $n\to\infty$,
     \:  $T_2=o_P(1).$
\end{proposition}

\begin{proposition} \label{prop: uniform convergence of I3}
   Under the conditions in Theorem \ref{theorem:global estimator for multivariate}, we have that uniformly for all $(u,v,s)\in [0,T]^2\times \set{jh: 1\leq j\leq m}$, as $n\to\infty$,
    \: $  T_3=o_P(1).$
\end{proposition}

\begin{proposition} \label{prop: uniform convergence of I4}
   Under the conditions in Theorem \ref{theorem:global estimator for multivariate}, we have that uniformly for all $(u,v,s)\in [0,T]^2\times \set{jh: 1\leq j\leq m}$, as $n\to\infty$,
   \begin{align*}
      T_4&=-\sqrt{h}\left\{\sum_{j=1}^{s/h}R_1(u,v;s_j)w_{n,j}\left(\frac{u}{b_{X,j}},\infty\right)+R_2(u,v;s_j)w_{n,j}\left(\infty,\frac{v}{b_{Y,j}}\right)\right\}
+o_P(1)\\ &
=:\widetilde T_4 +o_P(1).
   \end{align*}
\end{proposition}

\begin{proposition} \label{prop: uniform convergence of I5}
   Under the conditions in Theorem \ref{theorem:global estimator for multivariate}, we have that uniformly for all $(u,v,s)\in [0,T]^2\times \set{jh: 1\leq j\leq m}$, as $n\to\infty$,
    \: $  T_5=o(1).$
\end{proposition}
By combining Propositions \ref{prop: uniform convergence of I1}--\ref{prop: uniform convergence of I5}, we obtain that uniformly for all $(u,v)\in [0,T]^2$ and $j=1,2,\ldots, m$, as $n\to\infty$, 
\begin{equation} \label{eq:intermediate expansion}
\sqrt{k}\suit{\widehat{I}_{R,UU}(u,v;jh)- I_R(u,v;jh)}=\widetilde T_1(u,v;jh)+\widetilde T_4(u,v;jh)+o_P(1).    
\end{equation}
Note that for any $s\in(0,mh]$, $\widehat I_{R, UU}(u,v;s)$ is defined as a linear interpolation of two consecutive $\widehat I_{R, UU}(u,v;jh),$ for some $  j=0, 1, \ldots, m$, whereas for $s\in(mh,1]$ it is an extrapolation based on $\widehat I_{R,UU}(u,v;(m-1)h)$ and $\widehat I_{R,UU}(u,v;mh)$.
From (\ref{eq:intermediate expansion}), we readily obtain that uniformly for all $(u,v,s)\in [0,T]^2\times[0,1]$, as $n\to\infty$,
\begin{equation*} \sqrt{k}\suit{\widehat{I}_{R,UU}(u,v;s)-  I_R(u,v;s)}=\widetilde T_1(u,v;s)+\widetilde T_4(u,v;s)+o_P(1),  
\end{equation*}
where $\widetilde T_1$ and $\widetilde T_4$ are defined on $[0,T]^2\times[0,1]$ using a linear interpolation or extrapolation analogous  to the definition of $\widehat I_{R, UU}(u,v;s)$. 

Finally, the weak convergence of $\widetilde T_1(u,v;s)+\widetilde T_4(u,v;s)$ is  given by the  following proposition, which completes the proof of Theorem \ref{theorem:global estimator for multivariate}.
\begin{proposition} \label{prop: uniform convergence of w}
       Under the conditions in Theorem  \ref{theorem:global estimator for multivariate},  as $n\to\infty$, \begin{eqnarray*}&&\!\!\!\!\!\!\!\!\!\!\widetilde T_1(u,v;s)+\widetilde T_4(u,v;s)
      \\&&\qquad
      \leadsto  W_I(u,v;s)-\int_{0}^s R_1(u,v;w)W_I(u, \infty; dw)-\int_{0}^s R_2(u,v;w)W_I( \infty, v;dw), \end{eqnarray*}
    where $\leadsto$ denotes weak convergence on $D_{[0,T]^2\times[0,1]}$  equipped with the Skorokhod distance.
\end{proposition}

\subsection{Proof of the Lemmas and Propositions}

\begin{proof}[Proof of Lemma \ref{lemma: delta_F, general s}]
   Firstly, recall that $\underline{c}_X\leq c_X(s)\leq \overline{c}_X$ for all $s\in [0,1]$. Condition \ref{eq:condition heteroskedastic extremes} implies that there exists $x_0<x^*$ such that for all $x>x_0$ and $1\leq i\leq n$
   $$\frac{\underline{c}_X}{2}\leq \frac{1-F_i(x)}{1-F_0(x)}\leq 2\overline{c}_X.$$
   Since $\min_{1\leq j\leq m} \underline{H}_{F,j}(kT/n) =\min_{1\leq i\leq n} H_{F,i}(kT/n)\to x^*$ as $n\to\infty$,  we can thus replace $x$ by $\underline{H}_{F,j}(kT/n)$ uniformly over $j$, for sufficiently large $n$. This yields that for sufficiently large $n$, uniformly for all $1\leq j\leq m$ and for all $i\in I_j$,
   $$
   \frac{1}{1-F_0(\underline{H}_{F,j}(kT/n))}\geq \frac{\underline{c}_X}{2}\frac{1}{1-F_i(\underline{H}_{F,j}(kT/n))}\geq\frac{\underline{c}_X}{2\overline{c}_X}\frac{1}{\frac{1-F_i}{c_X(i/n)}\suit{\underline{H}_{F,j}(kT/n)}},
   $$
   which implies that
   \begin{equation} \label{eq:temp4}
      \frac{1}{1-F_0(\underline{H}_{F,j}(kT/n))}\geq \frac{\underline{c}_X}{2\overline{c}_X}\frac{1}{\min_{i\in I_j}\frac{1-F_i}{c_X(i/n)}\suit{\underline{H}_{F,j}(kT/n)}}=\frac{\underline{c}_X}{2\overline{c}_X T}\frac{n}{k}\, .
   \end{equation}
   Here in the last step, we use the fact that $\underline{H}_{F,j}=\min_{i\in I_j}H_{F,i}(x)$ is the left-continuous inverse function of $\min_{i\in I_j}\suit{\frac{1-F_i}{c_X(i/n)}}$.
   
   Next, by applying Condition \ref{eq:condition heteroskedastic extremes} to any $i,l\in I_j$, we get that as $x\to\infty$,
   $$\frac{1-F_i(x)}{1-F_l(x)}=\frac{c_X\suit{\frac{i}{n}}}{c_X\suit{\frac{l}{n}}}\suit{1+O\suit{A_X\suit{\frac{1}{1-F_0(x)}}}}.$$
   We replace $x$ by $\underline{H}_{F,j}(ku/n)$ for $u\in[0,T]$ in the equation above to obtain that as $n\to\infty$,    
   \begin{equation} \label{eq:temp5}
   \frac{1-F_i(\underline{H}_{F,j}(ku/n))}{1-F_l(\underline{H}_{F,j}(ku/n))}=\frac{c_X\suit{\frac{i}{n}}}{c_X\suit{\frac{l}{n}}}\suit{1+O\suit{A_X\suit{\frac{1}{1-F_0(\underline{H}_{F,j}(ku/n))}}}}.
   \end{equation}
   Note that $A_X$ is eventually non-increasing. Together with \eqref{eq:temp4}, we have that for  large $n$,
   $$A_X\suit{\frac{1}{1-F_0(\underline{H}_{F,j}(ku/n))}}\leq A_X\suit{\frac{1}{1-F_0(\underline{H}_{F,j}(kT/n))}}\leq A_X\suit{\frac{\underline{c}_X}{2\overline{c}_X T}\frac{n}{k}},$$
   uniformly for all $1\leq j\leq m$.
   Since the sequence $k$ satisfies Condition \ref{eq:condition k extra}, we can further write \eqref{eq:temp5} as
   $$\frac{1-F_i(\underline{H}_{F,j}(ku/n))}{1-F_l(\underline{H}_{F,j}(ku/n))}=\frac{c_X\suit{\frac{i}{n}}}{c_X\suit{\frac{l}{n}}}\suit{1+o\suit{\frac{1}{\sqrt{k}}}},$$
   where the term $o\suit{\frac{1}{\sqrt{k}}}$ is uniform for all $u\in[0,T]$, $i,l\in I_j$ and $1\leq j\leq m$.

   By rewriting the equation as 
    $$\frac{1-F_i(\underline{H}_{F,j}(ku/n))}{c_X\suit{\frac{i}{n}}}=\frac{1-F_l(\underline{H}_{F,j}(ku/n))}{c_X\suit{\frac{l}{n}}}\suit{1+o\suit{\frac{1}{\sqrt{k}}}},$$
   and taking the minimum over all $l\in I_j$ on both sides, we proved the lemma.
\end{proof}

\begin{proof}[Proof of Lemma \ref{lemma: asymptotics for the quantiles}]
We prove the first statement in the lemma. The second statement holds analogously.

Let $\Gamma_{j}(u):=\frac{1}{nh}\sum_{i\in I_j}\ll{U_i< u}$, the empirical distribution function based on $nh$ independent random variables $\set{U_i: i\in I_j}$. 
Obviously, $\E \Gamma_{j}(u)=\frac{1}{nh}\sum_{i\in I_j}c_X(i/n)u=b_{X,j}u. $  Let 
$\alpha_j(u)=\sqrt{nh}(\Gamma_j(u)-b_{X,j}u)$, for $u\geq 0$, be the corresponding  empirical process. Finally, denote  $w_{X,j}(u)=\sqrt{\frac{n}{k}} \alpha_j(ku/n)$ for  $u\geq 0$ as the corresponding tail empirical process. 
 
Then for large  $a>0$,
\begin{align*}
     &\P\left(\max_{1\leq j\leq m}\sup_{u\in[0,T]}\abs{b_{X,j}\frac{n}{k}U^{(j)}_{\lfloor khu \rfloor+1,nh}-u}\geq a\sqrt{\frac{m\log n}{k}}\right)\\
      =&\P\left(\max_{1\leq j\leq m}\sup_{u\in[0,T]}\abs{b_{X,j}\frac{n}{k}\Gamma^{-1}_{j}\suit{\frac{ku}{n}}-u}\geq a\sqrt{\frac{m\log n}{k}}\right)\\
      \leq &\P\left(\max_{1\leq j\leq m}\sup_{u\in[0,2T]}\abs{\frac{n}{k}\Gamma_{j}\suit{\frac{ku}{b_{X,j}n}}-u}\geq a\sqrt{\frac{m\log n}{k}}\right)\\
      = &\P\left(\max_{1\leq j\leq m}\sup_{u\in[0,2T]}\abs{w_{X,j}\suit{\frac{u}{b_{X,j}}}}\geq a\sqrt{\log n}\right)\\
      \leq & \sum_{j=1}^{m}\P\left(\sup_{u\in[0,2T]}\abs{w_{X,j}\suit{\frac{u}{b_{X,j}}}}\geq a\sqrt{\log n}\right)\\
      \leq&  \sum_{j=1}^{m} \P\left(\sup_{u\in[0,2kT/n]}\abs{ \alpha_j\suit{\frac{u}{b_{X,j}}}}\geq a\sqrt{\frac{k\log n}{n}}\right).
     \end{align*}
To further bound the probabilities above, we apply Corollary 2.9 in \citet{a84} with the upper bound therein  $16\exp(-M^2/(4\alpha))$, where $M=a\sqrt{\frac{k\log n}{n}}$ and $\alpha=2kT/n$ (which does not depend on $j$). Here we use the fact that $kh/\log n \to \infty$ as $n\to\infty$. Then the statement is proved by taking $a^2>8T$: as $n\to\infty$,
\begin{align*}
   &\P\left(\max_{1\leq j\leq m}\sup_{u\in[0,T]}\abs{b_{X,j}\frac{n}{k}U^{(j)}_{\lfloor khu \rfloor+1,nh}-u}\geq a\sqrt{\frac{m\log n}{k}}\right)\\
     \leq & 16 \sum_{j=1}^{m}  \exp\left(-\frac{ a^2k (\log n) n}{4 n \cdot 2kT}    \right) 
\leq  16 m \exp\left(-\frac{ a^2
\log n }{8T}    \right)\to 0.
   \end{align*} 
   
   The last two statements are proved by  relabeling $k$ with $k^{1-b}$ in the proof. Here we use the fact that  $k^{1-b}h/\log n \to \infty$ as $n\to\infty$.
\end{proof}

\begin{proof}[Proof of Proposition \ref{prop: uniform convergence of I1}]
Let $\varepsilon>0$. Then
\begin{align*}
   & \P(\max_{s \in\set{h, 2h, \ldots, mh}} \sup_{0\leq u,v\leq T} |T_1(u,v;s)-\widetilde T_1(u,v;s)|\geq \varepsilon )\\
\leq & \P\left(\sum_{j=1}^{m} \sup_{0\leq u,v\leq T} \left|w_{n,j}(\hat u_j,\hat v_j)-w_{n,j}\suit{\frac{u}{b_{X,j}}, \frac{v}{b_{Y,j}}}\right|\geq \varepsilon\sqrt{m} \right)\\
\leq & 
\sum_{j=1}^{m}  \P\left( \sup_{0\leq u,v\leq T}\sqrt{\frac{k}{n}} \, \left|w_{n,j}(\hat u_j,\hat v_j)-w_{n,j}\suit{\frac{u}{b_{X,j}}, \frac{v}{b_{Y,j}}}\right|\geq \varepsilon\sqrt{\frac{k}{nm}} \right).
\end{align*}
 We apply Corollary 2.9 in \citet{a84} again with the upper bound therein  $16\exp(-M^2/(4\alpha))$, $M=\varepsilon\sqrt{\frac{k}{nm}}$ and $\alpha=a\sqrt{km\log n}/n$; here $a$ is large. 
This yields that the sum of the  $m$  probabilities above is further bounded by
$$
16m\exp\left(-\frac{\varepsilon^2k\cdot n}{4nm\cdot a\sqrt{km\log n}}\right)+o(1)
\leq16\exp\left(-\left(\frac{\varepsilon^2k^{1/2}}{4am^{3/2}(\log n)^{1/2}}-\log n\right) \right)+o(1).$$ 
In order to apply Corollary 2.9 in \citet{a84} we have used 
$k/(m^3\log^3 n) =kh^3/\log^3 n\to \infty$, as $n\to \infty$, see Condition \ref{eq:condition on k}.  This assumption 
also yields that the latter probability bound tends to 0.
\end{proof}

\begin{proof}[Proof of Proposition \ref{prop: uniform convergence of I2}]
   Following Condition \ref{eq:condition k extra}, we can replace all terms in $T_2$ of the form
  $  \frac{n}{k}C_i\suit{\frac{k}{n}\,\cdot, \frac{k}{n} \,\cdot  }$
   by $R(\cdot, \cdot; i/n)$ uniformly, with negligible error. That means, it remains to show that uniformly
   \begin{align*}
      &\frac{1}{n}\sum_{j=1}^{s/h}\sum_{i\in I_j}R\suit{c_X\suit{\frac{i}{n}}\hat u_j, c_Y\suit{\frac{i}{n}}\hat v_j;\frac{i}{n}}\\    
    &=\frac{1}{n}\sum_{j=1}^{s/h}\sum_{i\in I_j}R\suit{c_X\suit{\frac{i}{n}}\hat u_j,c_Y\suit{\frac{i}{n}}\hat v_j;s_j}
    +o_P\left(\frac{1}{\sqrt{k}}\right).
   \end{align*}
   Since $R''$, and hence $R'$, are  bounded, by some common constant  
    $C>0$, we have   that for all $1\leq j\leq m$ and for all $i\in I_j$,
$$ R\suit{u,v;\frac{i}{n}}\leq R(u,v;s_j)+\suit{\frac{i}{n}-s_j}R'(u,v;s_j)+C\suit{\frac{i}{n}-s_j}^2,$$
and a similar lower bound with  $C$ replaced by $-C$.
Consequently, using  $\sum_{i\in I_j}\suit{\frac{i}{n}-s_j}=\frac{h}{2}$,
\begin{align*}
  &\sum_{i\in I_j}R\suit{c_X\suit{\frac{i}{n}}\hat u_j, c_Y\suit{\frac{i}{n}}\hat v_j;\frac{i}{n}}-R\suit{c_X\suit{\frac{i}{n}}\hat u_j,c_Y\suit{\frac{i}{n}}\hat v_j;s_j}\\
  \leq &\,  C\left(\frac{h}{2}+\sum_{i\in I_j}\suit{\frac{i}{n}-s_j}^2\right)\leq 2 Cnh^3. 
\end{align*}
Observe that a lower bound can be established in a similar way. Hence, we have that as $n\to\infty$,
$T_2=\sqrt{k}O\left (\frac{1}{nh}nh^3\right)=O(\sqrt k h^2)\to 0,$
due to the condition  $kh^4\to 0$.
\end{proof}

\begin{proof}[Proof of Proposition \ref{prop: uniform convergence of I3}]
We will show that 
uniformly for all $(u,v,s)\in [0,T]^2\times 
\{jh: j=1, \ldots, m\}$, as $n\to\infty$,
$$\sqrt{k}\frac{1}{n}\sum_{j=1}^{s/h}\left|\sum_{i\in I_j}\left(R\suit{c_X\suit{\frac{i}{n}}\hat u_j,c_Y\suit{\frac{i}{n}}\hat v_j;s_j}
 -R\suit{b_{X,j}\hat u_j,b_{Y,j}\hat v_j;s_j}\right)\right|=o_P(1).$$
 Using the mean-value theorem and triangular inequalities, there exist  $(\check{u}_{j,i}, \check{v}_{j,i})$  between $(c_X(i/n)\hat u_j, c_Y(i/n)\hat v_j)$   and $(b_{X,j}\hat u_j, b_{Y,j}\hat v_j)$, for all $j=1, \dots, m$ and $i\in I_j$, such that
\begin{eqnarray*} &&\sqrt{k}\frac{1}{n}\sum_{j=1}^{s/h}\left|\sum_{i\in I_j}\left(R\suit{c_X\suit{\frac{i}{n}}\hat u_j,c_Y\suit{\frac{i}{n}}\hat v_j;s_j}
 -R\suit{b_{X,j}\hat u_j,b_{Y,j}\hat v_j;s_j}\right)\right|\\
&& \leq\sqrt{k}\frac{1}{n}\sum_{j=1}^{s/h}\left|\sum_{i\in I_j}R_1\suit{\check u_{j,i},\check v_{j,i};s_j}
\hat u_j\left(c_X\suit{\frac{i}{n}}-b_{X,j}\right)
\right|
\\&&\qquad + \sqrt{k}\frac{1}{n}\sum_{j=1}^{s/h}\left|\sum_{i\in I_j}R_2\suit{\check u_{j,i},\check v_{j,i};s_j}
\hat v_j\left(c_Y\suit{\frac{i}{n}}-b_{Y,j}\right)
\right|.
\end{eqnarray*}
   The terms with $R_1$ and those with $R_2$ can be handled similarly.  Therefore we confine ourselves  to the first expression on the right. 
  We have 
\begin{eqnarray*}
&& \sqrt{k}\frac{1}{n}\sum_{j=1}^{s/h}\left |\sum_{i\in I_j}R_1\suit{\check u_{j,i},\check v_{j,i};s_j}
\hat u_j\left(c_X\suit{\frac{i}{n}}-b_{X,j}\right)
\right|\\&&=
\sqrt{k}\frac{1}{n}\sum_{j=1}^{s/h}\left |\sum_{i\in I_j}\left(R_1\suit{b_{X,j}\hat u_j,b_{Y,j}\hat v_j;s_j}
\hat u_j\left(c_X\suit{\frac{i}{n}}-b_{X,j}\right)\right.\right.
\\&&\qquad+\left.\left.
\left(R_1\suit{\check u_{j,i},\check v_{j,i};s_j}-R_1\suit{b_{X,j}\hat u_j,b_{Y,j}\hat v_j;s_j}\right)
\hat u_j\left(c_X\suit{\frac{i}{n}}-b_{X,j}\right)\right)
\right|\\&&=
\sqrt{k}\frac{1}{n}\sum_{j=1}^{s/h}\left |\sum_{i\in I_j}
\left(R_1\suit{\check u_{j,i},\check v_{j,i};s_j}-R_1\suit{b_{X,j}\hat u_j,b_{Y,j}\hat v_j;s_j}\right)
\hat u_j\left(c_X\suit{\frac{i}{n}}-b_{X,j}\right)
\right|
\\&&\leq
\sqrt{k}\frac{1}{n}\sum_{j=1}^{s/h}\sum_{i\in I_j}\left |
\left(R_1\suit{\check u_{j,i},\check v_{j,i};s_j}-R_1\suit{b_{X,j}\hat u_j,b_{Y,j}\hat v_j;s_j}\right)
\hat u_j\left(c_X\suit{\frac{i}{n}}-b_{X,j}\right)
\right|.
\end{eqnarray*}
Here for the equality, we use the definition that $\sum_{i\in I_j}\left(c_X\suit{\frac{i}{n}}-b_{X,j}\right)=0$ for all $j$.
Now assume for all $j=1, \dots, s/h$, that  $b_{X,j}\hat u_j\leq b_{Y,j}\hat v_j$. The other $2^{s/h}-1$ situations can be handled in the same way. Then, using the homogeneity of order 0 of $R_1$ and its Lipschitz continuity (since the second coordinate is larger than $1/2$), the last bound is in turn bounded~by 
\begin{eqnarray*}
&&
\sqrt{k}\frac{1}{n}\sum_{j=1}^{s/h}\sum_{i\in I_j}\left|\left(
R_1\left(\frac{c_X\suit{\frac{i}{n}}\hat u_j}{b_{Y,j}\hat v_j},
\frac{c_Y\suit{\frac{i}{n}}\hat v_j}{b_{Y,j}\hat v_j};s_j\right)-R_1\suit{\frac{b_{X,j}\hat u_j}{ b_{Y,j}\hat v_j},1;s_j}\right)
\hat u_j\left(c_X\suit{\frac{i}{n}}-b_{X,j}\right)
\right|\\&&
\leq \sqrt{k}\frac{1}{n}\sum_{j=1}^{s/h}\sum_{i\in I_j} \frac{K
\left(\hat u_j^2\left(c_X\suit{\frac{i}{n}}-b_{X,j}\right)^2+\hat u_j\hat v_j\left|c_X\suit{\frac{i}{n}}-b_{X,j}\right|\cdot
\left|c_Y\suit{\frac{i}{n}}-b_{Y,j}\right|\right)}{b_{Y,j}\hat v_j}
\\&&
\leq \sqrt{k}\frac{1}{n}\sum_{j=1}^{s/h}\sum_{i\in I_j} \frac{K
\left(\hat u_j\left(c_X\suit{\frac{i}{n}}-b_{X,j}\right)^2+\hat v_j\left|c_X\suit{\frac{i}{n}}-b_{X,j}\right|\cdot
\left|c_Y\suit{\frac{i}{n}}-b_{Y,j}\right|\right)}{b_{X,j}}
\\&&
=O_P\left(\sqrt{k} \frac{1}{n} \frac{1}{h} \cdot nh \cdot h^2 \right)=o_P(1).
\end{eqnarray*}
\end{proof}
\begin{proof}[Proof of Proposition \ref{prop: uniform convergence of I4}]
Write  $\hat u_{j0}=\frac{n}{k}U^{(j)}_{\lfloor khu \rfloor+1,nh}$ and define $\hat v_{j0}$ correspondingly. By the Lipschitz continuity  of $R$ and  Lemma \ref{lemma: delta_F, general s} we have uniformly 
 \begin{eqnarray*}&&\sqrt{k} h\sum_{j=1}^{s/h}\left(
R\suit{b_{X,j}\hat u_j,b_{Y,j}\hat v_j;s_j}-R\suit{b_{X,j}\hat u_{j0},b_{Y,j}\hat v_{j0};s_j}\right)\\
&&=\sqrt{k} h\sum_{j=1}^{s/h}\left(
R\suit{b_{X,j}\hat u_{j0}(1+\Delta),b_{Y,j}\hat v_{j0}(1+\Delta);s_j}-R\suit{b_{X,j}\hat u_{j0},b_{Y,j}\hat v_{j0};s_j}\right)=o_P(1).
    \end{eqnarray*}
By the mean-value theorem we have
 \begin{eqnarray*}&&\sqrt{k} h\sum_{j=1}^{s/h}\left(
R\suit{b_{X,j}\hat u_{j0},b_{Y,j}\hat v_{j0};s_j}-R(u,v;s_j)\right)\\
    &&=\sqrt{k} h\sum_{j=1}^{s/h}\left(R_1(\check u_j,
      \check v_j;s_j)(b_{X,j}\hat u_{j0}-u)+R_2(\check u_j,\check v_j;s_j)(b_{Y,j}\hat v_{j0}-v)\right),
   \end{eqnarray*}
   with  $(\check{u}_j, \check{v}_j)$  between $(b_{X,j}\hat u_{j0},b_{Y,j}\hat v_{j0})$  and $(u,v)$, $j=1, \dots, s/h$. The terms with $R_1$ and those with $R_2$ on the right are similar.  Therefore we proceed in this proof with only the $R_1$-expression and omit the 
   $R_2$-expression, also for $\widetilde T_4$. 

We first show that uniformly
$$\sqrt{k} h\sum_{j=1}^{s/h}(R_1(\check u_j,
      \check v_j;s_j)-R_1(u,v;s_j))(b_{X,j}\hat u_{j0}-u)=o_P(1). $$
      We consider three regions for $(u,v)$. We begin with   $[0,1/k^{1/3}]\times [0,T]$.   Note that $0\leq R_1\leq 1$. Now, uniformly for $u\in [0,1/k^{1/3}]$,
      $$|\sqrt{k} h\sum_{j=1}^{s/h}(R_1(\check u_j,
      \check v_j;s_j)-R_1(u,v;s_j))(b_{X,j}\hat u_{j0}-u)|
     \leq  \sqrt{k}h\sum_{j=1}^{s/h}|b_{X,j}\hat u_{j0}-u|.$$
To bound this we use the second half of Lemma \ref{lemma: asymptotics for the quantiles} with $b=1/3$, which yields an uniform  upper bound as
$$\sqrt{k}h \cdot m\cdot O_P\left(\frac{(\log n)^{1/2}}{h^{1/2}k^{2/3}}\right)=O_P\left(\left(\frac{(\log n)^{3}}{h^{3}k }\right)^{1/6}\right)=o_p(1).$$
Next we consider  $(u,v)\in ([1/k^{1/6}, T] \times [0,T]) \cup 
([0, T] \times [1/k^{1/6},T]) $. Then using that $R_1$ is homogeneous of order 0 and Lipschitz continuous, we have with probability tending to~1:
\begin{align*}|R_1(\check u_j,\check v_j;s_j)-R_1(u,v;s_j)|
&=|R_1(k^{1/6}\check u_j,k^{1/6}\check v_j;s_j)-R_1(k^{1/6}u,k^{1/6}v;s_j)|\\
&\leq Kk^{1/6}(|\check u_j-u|+|\check v_j-v|).\end{align*}
Hence,  by the first half of Lemma \ref{lemma: asymptotics for the quantiles}, uniformly on the region,\begin{eqnarray*}&&\sqrt{k} h
\sum_{j=1}^{s/h}
(R_1(\check u_j,
      \check v_j;s_j)-R_1(u,v;s_j))(b_{X,j}\hat u_{j0}-u)
\\
&&\qquad= \sqrt{k}h\cdot m\cdot Kk^{1/6}O_P\left(\frac{\log n }{hk}\right)=O_P\left(\left(\frac{(\log n)^{3}}{h^{3}k }\right)^{1/3}\right)=o_p(1).
\end{eqnarray*}
It remains to consider  $(u,v)\in [1/k^{1/3}, 1/k^{1/6}] \times [0,1/k^{1/6}]$.  With probability tending to 1,
\begin{align*}|R_1(\check u_j,\check v_j;s_j)-R_1(u,v;s_j)|
&=|R_1(k^{1/3}\check u_j,k^{1/3}\check v_j;s_j)-R_1(k^{1/3}u,k^{1/3}v;s_j)|\\&
\leq Kk^{1/3}(|\check u_j-u|+|\check v_j-v|).\end{align*}
      Hence  by the second half of Lemma \ref{lemma: asymptotics for the quantiles} with $b=1/6$, uniformly on the region,\begin{eqnarray*}
         &\sqrt{k} h\sum_{j=1}^{s/h}(R_1(\check u_j,
      \check v_j;s_j)-R_1(u,v;s_j))(b_{X,j}\hat u_{j0}-u)\\
&\qquad= \sqrt{k}h\cdot m\cdot Kk^{1/3}O_P\left(\frac{\log n }{hk^{7/6}}\right)=O_P\left(\left(\frac{(\log n)^{3}}{h^{3}k }\right)^{1/3}\right)=o_p(1).
\end{eqnarray*}

It remains to 
show that uniformly $$\sqrt{k}h\sum_{j=1}^{s/h}R_1(u,v;s_j)(b_{X,j}\hat u_{j0}-u)
 +\sqrt{h}\sum_{j=1}^{s/h}R_1(u,v;s_j)w_{n,j}\left(\frac{u}{b_{X,j}},\infty\right)        =o_P(1). $$
For this it suffices to show that
\begin{equation}\label{ists}\sqrt{\frac{1}{h}}\max_{1\leq j\leq m} \sup_{0\leq u \leq T}\left|
\sqrt{kh} (b_{X,j}\hat u_{j0}-u)+w_{n,j}(u/b_{X,j},\infty)\right|=o_P(1).\end{equation}
Note that, almost surely,
$$\sqrt{\frac{1}{h}}\max_{1\leq j\leq m} \sup_{0\leq u \leq T}\left|
\sqrt{kh}\left (\frac{n}{k}\Gamma_{j}(\Gamma_j^{-1}(uk/n))-u\right)\right|\leq\sqrt{\frac{1}{h}}\sqrt{kh}\frac{1}{k}=\frac{1}{\sqrt{k}}.$$
By replacing $u$ in the first term in (\ref{ists}) with $(n/k) (\Gamma_j(\Gamma_j^{-1}(uk/n))$, in turn it now suffices to show that 
$$\sqrt{1/h}\max_{1\leq j\leq m} \sup_{0\leq u \leq T}\left|-w_{n,j}((n/k) \Gamma_j^{-1}(uk/n),\infty)
+w_{n,j}(u/b_{X,j},\infty)\right|=o_P(1).
$$
Let $\varepsilon>0$. We have  for large $a>0$
\begin{eqnarray*}&&\mathbb{P}\left(\max_{1\leq j\leq m} \sup_{0\leq u \leq T}\left|-w_{n,j}((n/k)\Gamma_j^{-1}(uk/n),\infty)
+w_{n,j}(u/b_{X,j},\infty)\right|\geq \varepsilon\sqrt{h}\right)\\
&&\leq \mathbb{P}\!\left(\max_{1\leq j\leq m} \sup_{\substack{0\leq u \leq T,\\ |v-u|\leq a\sqrt{(\log n)/(kh)} }}\!\left|w_{n,j}(u/b_{X,j},\infty)
-w_{n,j}(v/b_{X,j},\infty)\right|\geq \varepsilon\sqrt{h}\right)+o(1)\\
&&\leq \sum_{j=1}^m 
\mathbb{P}\!\left(\sup_{\substack{0\leq u \leq T,\\ |v-u|\leq a\sqrt{(\log n)/(kh)} }}\!\left|w_{n,j}(u/b_{X,j},\infty)
-w_{n,j}(v/b_{X,j},\infty)\right|\geq \varepsilon\sqrt{h}\right)+o(1)\\
&&= \sum_{j=1}^m 
\mathbb{P}\!\left(\sup_{\substack{0\leq x \leq kT/n,\\ |y-x|\leq (a/n)\sqrt{(k\log n)/h} }}\!\left|\alpha_j(x/b_{X,j})
-\alpha_j(y/b_{X,j})\right|\geq \varepsilon\sqrt{hk/n}\right)+o(1). 
\end{eqnarray*}
Now using again  Corollary 2.9 in \citet{a84} we
have very similarly as in the proof of Proposition  \ref{prop: uniform convergence of I1} that this last expression tends to 0. \end{proof}

\begin{proof}
[Proof of Proposition \ref{prop: uniform convergence of I5}]
 From Condition \ref{eq:condition partial derivatives of R}
  we have for some $C>0$ that for all $1\leq j\leq m$ and all $t\in[s_j-h,s_j+h]$
   \begin{align*}
   &R(u,v;s_j)+(t-s_j)R'(u,v;s_j)-C(t-s_j)^2\\
\leq &R(u,v;t)\leq R(u,v;s_j)+(t-s_j)R'(u,v;s_j)+C(t-s_j)^2.\end{align*}
  Hence for $s\in\{jh:1\leq j\leq m\}$
   \begin{align*}
      I_R(u,v;s)=&\sum_{j=1}^{s/h}\int_{(j-1)h}^{jh}R(u,v;t)dt
      \\
      \leq &\sum_{j=1}^{s/h}\int_{(j-1)h}^{jh}\left[R(u,v;s_j)+(t-s_j)R'(u,v;s_j)+C(t-s_j)^2\right]dt\\
\leq&h\sum_{j=1}^{s/h}R(u,v;s_j)+C\frac{s}{h}h^3
    \leq h\sum_{j=1}^{s/h}R(u,v;s_j)
      +Ch^2.
   \end{align*}
   A similar relation holds for the lower bound on $I_R(u,v;s)$. Therefore, uniformly for all $s\in\set{jh:j=1,2,\ldots,m}$,
   $ T_5=O(\sqrt{k}h^2)\to 0$  as $n\to\infty$,
  by  Condition \ref{eq:condition on k}.
   \end{proof}

   \begin{proof}[Proof of Proposition \ref{prop: uniform convergence of w}] We begin with showing that the finite-dimensional distributions of $\widetilde T_1+\widetilde T_4$  converge in distribution  to those of the limiting process, briefly denoted by $L=L(u,v;s)$.  For this we  first observe that the  two implicitly mentioned interpolation terms 
   of $\widetilde T_1+\widetilde T_4$ converge in distribution  to the 0-function uniformly and hence can be ignored, i.e., we need to show the convergence in distribution of the   finite-dimensional distributions  of the sum of the main expressions of $\widetilde T_1$ and $\widetilde T_4$, which 
   can be written as 
\begin{eqnarray*}
&&L_n(u,v;s):=\frac{1}{\sqrt{k}}\sum_{i=1}^{\lfloor s/h \rfloor nh}\Bigl(1_{[U_i\leq ku/(b_{X,j}n), V_i\leq kv/(b_{Y,j}n)]}-C_i\left(\frac{c_X(i/n)ku}{b_{X,j}n}, \frac{c_Y(i/n)kv}{b_{Y,j}n}\right)\Bigr)\\
&&\qquad-\frac{1}{\sqrt{k}}\sum_{i=1}^{\lfloor s/h \rfloor nh}R_1(u,v;s_j) \left(1_{[U_i\leq ku/(b_{X,j}n)]}-c_X(i/n)ku/(b_{X,j}n)\right)
\\&&\qquad-\frac{1}{\sqrt{k}}\sum_{i=1}^{\lfloor s/h \rfloor nh}R_2(u,v;s_j) \left(1_{[V_i\leq kv/(b_{Y,j}n)]}-c_Y(i/n)kv/(b_{Y,j}n)\right),
\end{eqnarray*}
where  for each $i=1, \ldots, n$, the index $j=j(i)$ is defined by $i\in I_j$.

Using the  Cram\'er-Wold device it is sufficient to show
that $\sum_{l=1}^r \lambda_l L_n(u_l,v_l;t_l)$ converges in distribution to $\sum_{l=1}^r \lambda_l L(u_l,v_l;t_l)$. This will boil down to showing that  $Var(\sum_{l=1}^r \lambda_l L_n(u_l,v_l;t_l)) \to Var(\sum_{l=1}^r \lambda_l L(u_l,v_l;t_l))$. Calculating these variances is rather cumbersome, but essentially the same as calculating the variances when  $r=2$, since all possible types of individual variances and covariances  are already represented then. Therefore for the sake of presentation  we restrict ourselves to the case $r=2$.
Hence we need to show that that $$\lambda_1 L_n(u_1,v_1;s)+ \lambda_2 L_n (u_2,v_2;t)\stackrel{d}{
\to}\lambda_1 L(u_1,v_1;s)+ \lambda_2 L(u_2,v_2;t).$$
For this we use the Lindeberg CLT. Since indicators and the partial derivatives of a tail copula are both bounded by 1, the Lindeberg condition is trivially satisfied.
So it remains to show that  $$Var(\lambda_1 L_n(u_1,v_1;s)+ \lambda_2 L_n(u_2,v_2;t))\to Var(\lambda_1 L(u_1,v_1;s)+ \lambda_2 L(u_2,v_2;t)),$$ and for this it sufficient to prove that for all
$(u_1,v_1,s), (u_2,v_2,t) \in [0,T]^2\times [0,1]$,
$$Cov( L_n(u_1,v_1;s), L_n(u_2,v_2;t))\to Cov (L(u_1,v_1;s),L(u_2,v_2;t)).$$ 
The latter covariance
 is equal to
 \begin{eqnarray*}
&&
\int_0^{s\wedge t}R(u_1\wedge u_2,v_1\wedge v_2; w)dw\\
&&\qquad -\int_0^{s\wedge t}R_1(u_2,v_2; w)R(u_1\wedge u_2, v_1; w)dw-\int_0^{s\wedge t}R_2(u_2,v_2; w)R(u_1, v_1\wedge v_2; w)dw\\
&&\qquad  -\int_0^{s\wedge t}R_1(u_1,v_1; w)R(u_1\wedge u_2, v_2; w)dw-\int_0^{s\wedge t}R_2(u_1,v_1; w)R(u_2, v_1\wedge v_2; w)dw\\
&& +(u_1 \wedge u_2)\int_0^{s\wedge t}R_1(u_1,v_1; w)R_1(u_2, v_2; w)dw +(v_1 \wedge v_2)\int_0^{s\wedge t}R_2(u_1,v_1; w)R_2(u_2, v_2; w)dw\\
&&\!\!  +\!\int_0^{s\wedge t}\!\!\!R_1(u_1,v_1; w)R_2(u_2, v_2; w)R(u_1,v_2; w)dw+\int_0^{s\wedge t}\!\!R_1(u_2,v_2; w)R_2(u_1, v_1; w)R(u_2,v_1; w)dw.
\end{eqnarray*}
We would like to show  for each of these 9 terms, that the   corresponding term based on $L_n$ (see below) converges to this term. 
We will actually  show the convergence to the $1^{st}$ term and to $8^{th}$ term. The convergence to the other terms is similar (term 9) or somewhat easier to show (terms 2-7) than the convergence to the $8^{th}$ term.


Regarding the $1^{st}$ term we need to show that
\begin{eqnarray*}&&E\left[\frac{1}{\sqrt{k}}\sum_{i=1}^{\lfloor s/h \rfloor nh}
\Bigl(1_{[U_i\leq ku_1/(b_{X,j}n), V_i\leq kv_1/(b_{Y,j}n)]}
 -C_i\left(\frac{c_X(i/n)ku_1}{b_{X,j}n},\frac{ c_Y(i/n)kv_1}{b_{Y,j}n}\right)\Bigr) \right.
\\&&\left.\cdot\,\frac{1}{\sqrt{k}}\sum_{i=1}^{\lfloor t/h \rfloor nh}
\Bigl(1_{[U_i\leq ku_2/(b_{X,j}n), V_i\leq kv_2/(b_{Y,j}n)]}
-C_i\left(\frac{c_X(i/n)ku_2}{b_{X,j}n}, \frac{c_Y(i/n)kv_2}{b_{Y,j}n}\right)\Bigr)
\right]
\\&&\qquad\qquad\to \int_0^{s\wedge t}R(u_1\wedge u_2,v_1\wedge v_2; w)dw.
\end{eqnarray*}
The left-hand side of this expression is equal to
\begin{eqnarray*}&&\frac{1}{k}\sum_{i=1}^{\lfloor  (s\wedge t)/h \rfloor nh}\E\Bigl(1_{[U_i\leq ku_1/(b_{X,j}n), V_i\leq kv_1/(b_{Y,j}n)]}-C_i\left(\frac{c_X(i/n)ku_1}{b_{X,j}n},\frac{ c_Y(i/n)kv_1}{b_{Y,j}n}\right)\Bigr)\\&&\qquad\qquad \cdot
\Bigl(1_{[U_i\leq ku_2/(b_{X,j}n), V_i\leq kv_2/(b_{Y,j}n)]}-C_i\left(\frac{c_X(i/n)ku_2}{b_{X,j}n}, \frac{c_Y(i/n)kv_2}{b_{Y,j}n}\right)\Bigr)\\
&&=\frac{1}{n}\sum_{i=1}^{\lfloor  (s\wedge t)/h \rfloor nh}\Bigl[(n/k)C_i\left(\frac{c_X(i/n)k(u_1\wedge u_2)}{b_{X,j}n},
\frac{c_Y(i/n)k(v_1\wedge v_2)}{b_{Y,j}n}\right)
\\&&\qquad-(n/k)C_i\left(\frac{c_X(i/n)ku_1}{b_{X,j}n},
\frac{c_Y(i/n)kv_1}{b_{Y,j}n}\right)\cdot C_i\left(\frac{c_X(i/n)ku_2}{b_{X,j}n},
\frac{c_Y(i/n)kv_2}{b_{Y,j}n}\right)\Bigr]\\[5 pt]
&&=\left(\frac{1}{n}\sum_{i=1}^{\lfloor  (s\wedge t)/h \rfloor nh}R(c_X(i/n)(u_1\wedge u_2)/b_{X,j}, c_Y(i/n)(v_1\wedge v_2)/b_{Y,j}; i/n)\right)+o(1)\\
&&=\left(\frac{1}{n}\sum_{i=1}^{\lfloor  (s\wedge t)/h \rfloor nh}R(u_1\wedge u_2, v_1\wedge v_2; i/n)\right)+o(1)
\to \int_0^{s\wedge t}R(u_1\wedge u_2,v_1\wedge v_2; w)dw.
\end{eqnarray*}

Now, as mentioned above, we conclude with the   convergence to the $8^{th}$ term:
\begin{eqnarray*}
   &&\E\left[\frac{1}{\sqrt{k}}\sum_{i=1}^{ \lfloor \frac{s}{h} \rfloor nh}R_1(u_1,v_1;s_j) \left(1_{[U_i\leq ku_1/(b_{X,j}n)]}-\frac{c_X(i/n)ku_1}{b_{X,j}n}\right)\right.\\&&\qquad\qquad\left.
\cdot
\frac{1}{\sqrt{k}}\sum_{i=1}^{ \lfloor \frac{t}{h} \rfloor nh}R_2(u_2,v_2;s_j) \left(1_{[V_i\leq kv_2/(b_{Y,j}n)]}-\frac{c_Y(i/n)kv_2}{b_{Y,j}n}\right)\right]\\
&&=\frac{1}{k}\sum_{i=1}^{\lfloor \frac{s\wedge t}{h} \rfloor nh}R_1(u_1,v_1;s_j)R_2(u_2,v_2;s_j)\\
&&\qquad\qquad\cdot \,\E\left(1_{[U_i\leq \frac{ku_1}{b_{X,j}n}]}-\frac{c_X(i/n)ku_1}{b_{X,j}n}\right)\left(1_{[V_i\leq \frac{kv_2}{b_{Y,j}n}]}-\frac{c_Y(i/n)kv_2}{b_{Y,j}n}\right)\\
&&=\frac{1}{n}\sum_{i=1}^{\lfloor \frac{s\wedge t}{h} \rfloor nh}R_1(u_1,v_1;s_j)R_2(u_2,v_2;s_j)\cdot\frac{n}{k}C_i\left(\frac{c_X(i/n)ku_1}{b_{X,j}n},\frac{c_Y(i/n)kv_2}{b_{Y,j}n}\right)+o(1)\\
&&=\frac{1}{n}\sum_{i=1}^{\lfloor \frac{s\wedge t}{h} \rfloor nh}R_1(u_1,v_1;s_j)R_2(u_2,v_2;s_j)\cdot R\left(\frac{c_X(i/n)u_1}{b_{X,j}}, \frac{c_Y(i/n)v_2}{b_{Y,j}}, s_j\right)+o(1)\\
&&=\frac{1}{n}\sum_{i=1}^{\lfloor \frac{s\wedge t}{h} \rfloor nh}R_1(u_1,v_1;s_j)R_2(u_2,v_2;s_j)R(u_1, v_2; s_j)+o(1)\\
&&=h\sum_{j=1}^{\lfloor  (s\wedge t)/h \rfloor}R_1(u_1,v_1;s_j)R_2(u_2,v_2;s_j)R(u_1, v_2; s_j)+o(1)\\
&&\to\int_0^{s\wedge t}R_1(u_1,v_1; w)R_2(u_2, v_2; w)R(u_1,v_2; w)dw.
\end{eqnarray*}

This completes the proof of the convergence of the finite-dimensional distributions of~$L_n$ and hence of $\widetilde T_1+\widetilde T_4$.


For the proof of the asymptotic tightness of $\widetilde T_1+\widetilde T_4$ we use Theorem 3 in \citet{es21}; we refer the reader to statement of the theorem therein and note that certain conditions (21), (22), and (23) have to be verified. Write $j(s)= \lceil s/h \rceil \wedge m$. The $Z_{n,i}$ therein are
defined by 
\begin{eqnarray*}
&&\!\!\!  \!\!\! Z_{n,i}(u,v;s):=\frac{1}{\sqrt{k}}1_{[U_i\leq \frac{ku}{b_{X,j}n}, V_i\leq \frac{kv}{b_{Y,j}n}]}\Bigl(1_{[i\leq (j(s)-1)nh]}
+\left(\frac{s}{h}-(j(s)-1)\right)
1_{[(j(s)-1)nh <i \leq j(s)nh ]}\Bigr)
\\
&&\qquad -\frac{1}{\sqrt{k}}R_1(u,v;s_j) 1_{[U_i\leq \frac{ku}{b_{X,j}n}]}\Bigl(1_{[i\leq (j(s)-1)nh]}+\left(\frac{s}{h}-(j(s)-1)\right) 
1_{[(j(s)-1)nh <i \leq j(s)nh]}\Bigr)
\\&&\qquad-\frac{1}{\sqrt{k}}R_2(u,v;s_j) 1_{[V_i\leq \frac{kv}{b_{Y,j}n}]}\Bigl(1_{[i\leq (j(s)-1)nh]}+\left(\frac{s}{h}-(j(s)-1)\right) 
1_{[(j(s)-1)nh <i \leq j(s)nh]}\Bigr)\, \\
&&\qquad\qquad=:Z_{n,i,1}(u,v;s)-Z_{n,i,2}(u,v;s)-Z_{n,i,3}(u,v;s).
\end{eqnarray*}
 In the sequel we will tacitly use that 
 $$\lim_{n\to \infty}\max_{1\leq j\leq m} \max_{i\in I_j}\left|\frac{c_X(i/n)}{b_{X,j}} -1\right|=0$$ and a similar statement for the $Y$-coordinate.

The first condition, (21), follows easily since for all $(u,v,s)\in [0,T]^2 \times [0,1]$, $|Z_{n,i}(u,v;s)| \leq 4/\sqrt{k}\to 0$, as $n\to \infty$, meaning that the indicator in (21)  is 0, for large $n$.
For the conditions (22) and (23) we take the covering of  $\mathcal{F}=[0,T]^2\times [0,1]$  with elements $[m_1\varepsilon^4, (m_1+1)\varepsilon^4]\times [m_2\varepsilon^4, (m_2+1)\varepsilon^4]\times [m_3\varepsilon^2, (m_3+1)\varepsilon^2], m_1, m_2\in \{0, 1, \dots, \lfloor T/\varepsilon^4 \rfloor\},  m_3\in \{0, 1, \dots, \lfloor 1/\varepsilon^2 \rfloor\}.$
Clearly condition (23) is amply satisfied now. It remains to consider (22). Since $(a+b+c)^2\leq 3(a^2+ b^2+c^2)$, it suffices to show (22) separately  for each term in the definition of $Z_{n,i}$. We have to show upper bounds of the form  $c\varepsilon^2$, for some $c>0$,  for certain expressions. 

First consider $Z_{n,i,1}(u,v;s)$
For an element of the covering, a little hyperrectangle $H$, write $u_1, v_1, t_1$ for the lower bounds of $u,v,s$ in $H$ and  $u_2, v_2, t_2$ for the upper bounds.
Then \begin{eqnarray*}
&&\sum_{i=1}^n \E\sup_{(u,v,s),(\tilde u, \tilde v, \tilde s) \in H}|Z_{n,i,1}(u,v;s)-Z_{n,i,1}(\tilde u,\tilde v;\tilde  s)|^2\\
&&=\sum_{i=1}^n \E|Z_{n,i,1}(u_2,v_2;t_2)-Z_{n,i,1}(u_1,v_1;t_1)|^2\\
&&\leq \sum_{i=1}^n \E \left[\frac{1}{\sqrt{k}}1_{[\frac{ku_1}{n} < b_{X,j}U_i\leq \frac{ku_2}{n} \text{ or } \frac{kv_1}{n}< b_{Y,j}V_i\leq \frac{kv_2}{n}]}\right.
\\&&\qquad\qquad\qquad\cdot\left(1_{[i\leq (j(t_1)-1)nh]}
+(t_1/h-(j(t_1)-1)) 
1_{[(j(t_1)-1)nh <i \leq j(t_1)nh ]}\right)
\\
&&\qquad\qquad+\frac{1}{\sqrt{k}}1_{[b_{X,j}U_i\leq \frac{kT}{n}, b_{Y,j} V_i\leq \frac{kT}{n}]}\\&&\qquad\qquad\cdot
\left(1_{[(j(t_1)-1)nh<i\leq (j(t_2)-1)nh]}\right.
 +(t_2/h-(j(t_2)-1)) 
1_{[(j(t_2)-1)nh <i \leq j(t_2)nh ]}
\\&&\qquad\qquad\qquad\left.\left. - (t_1/h-(j(t_1)-1)) 
1_{[(j(t_1)-1)nh <i \leq j(t_1)nh}
\right)  
\right]^2\\
&&\leq 2\left(\frac{1}{k}\frac{k}{n} 3\varepsilon^4 \cdot 4n+\frac{1}{k}\frac{2kT}{n}\frac{4}{h}\varepsilon^2nh \right)\leq 8(3+2T)\varepsilon^2,
\end{eqnarray*}
where we have used $(a+b)^2\leq 2(a^2+ b^2)$ 
for the last inequality. This shows (22) for $Z_{n,i,1}$. It remains to show (22) for $Z_{n,i,2}$ and $Z_{n,i,3}$. Since these terms are very similar we can confine ourselves to $Z_{n,i,2}$ only.

Note that $s_j$ is determined by the index $i$.
Recall that for an element of the covering $H$, we write $u_1, v_1, t_1$ for the lower bounds of $u,v,s$ in $H$ and  $u_2, v_2, t_2$ for the upper bounds. 
Write $R_1^+=\sup_{(u,v,s)\in H} R_1(u,v; s_j)$
and  $R_1^-=\inf_{(u,v,s)\in H} R_1(u,v; s_j)$; note that $s$ plays no role when taking the sup or inf. Also recall that $R_1\leq 1$. We have
\begin{eqnarray*}
&&\sum_{i=1}^n \E\sup_{(u,v,s),(\tilde u, \tilde v, \tilde s) \in H}
|Z_{n,i,2}(u,v;s)-Z_{n,i,2}(\tilde u,\tilde v;\tilde s)|^2\\
&&\leq \sum_{i=1}^n \E \left(\frac{1}{\sqrt{k}} (R_1^+-R_1^-)
{1_{[b_{X,j} U_i\leq ku_1/n]}}\right.\\ &&\qquad\qquad\cdot\, \Bigl(1_{[i\leq (j(t_1)-1)nh]}
+(t_1/h-(j(t_1)-1)) 
1_{[(j(t_1)-1)nh <i \leq j(t_1)nh ]}\Bigr)\\
&&\qquad+\frac{1}{\sqrt{k}}1_{[\frac{ku_1}{n} <b_{X,j} U_i\leq \frac{ku_2}{n}]}\Bigl(1_{[i\leq (j(t_1)-1)nh]}+\left(\frac{t_1}{h}-(j(t_1)-1)\right) 
1_{[(j(t_1)-1)nh <i \leq j(t_1)nh ]}\Bigr)\\
&&+\frac{1}{\sqrt{k}}1_{[b_{X,j} U_i\leq \frac{kT}{n}]}
    \left(1_{[(j(t_1)-1)nh<i\leq (j(t_2)-1)nh]}+\left(\frac{t_2}{h}-(j(t_2)-1)\right)
1_{[(j(t_2)-1)nh <i \leq j(t_2)nh ]}\right.
\\&&\qquad\qquad\qquad\left. - (t_1/h-(j(t_1)-1)) 
1_{[(j(t_1)-1)nh <i \leq j(t_1)nh}
\right)  \big)^2.
\end{eqnarray*}
Using the Lipschitz continuity of $R_1$ and its homogeneity of order 0, we obtain for  $u_1\geq \varepsilon^2$, 
\begin{eqnarray*}&&R_1^+-R_1^-=\sup_{(u,v,s)\in H} R_1(1,v/u; s_j) -\inf_{(u,v,s)\in H} R_1(1,v/u; s_j)\\&&=\sup_{(u,v,s), (\tilde u, \tilde v, \tilde s)\in H} R_1(1,v/u; s_j) - R_1(1,\tilde v/\tilde u; s_j)\\
&&\leq\sup_{(u,v,s), (\tilde u, \tilde v, \tilde s)\in H} K\left|\frac{v}{u}-\frac{\tilde v}{\tilde u}\right|
=K\left(\frac{v_2}{u_1}-\frac{v_1}{ u_2}\right)\\
&&=\frac{K}{u_1}\left(v_2-v_1+v_1\left(\frac{u_2-u_1}{u_2}\right)\right)
\leq \frac{K}{u_1}(\varepsilon^4+T\varepsilon^4/\varepsilon^2)=\frac{K(T+1)}{\sqrt{u_1}}\varepsilon.
\end{eqnarray*}
Hence, taking $K\geq 1$,
\begin{eqnarray*}
\E(R_1^+-R_1^-)^21_{[b_{X,j} U_i\leq \frac{ku_1}{n}]}\leq \max\bigl(2k\varepsilon^2/n,\, 2k/n\cdot K^2(T+1)^2\varepsilon^2\bigr)\leq 2k/n\cdot K^2(T+1)^2\varepsilon^2.
\end{eqnarray*} 
Now using again $(a+b+c)^2\leq 3(a^2+ b^2+c^2)$, we find similarly as for $Z_{n,i,1}$ that
\begin{eqnarray*}
&&\sum_{i=1}^n \E\sup_{(u,v,s),(\tilde u, \tilde v, \tilde s) \in H}|Z_{n,i,2}(u,v;s)-Z_{n,i,2}(\tilde u,\tilde v;\tilde s)|^2\\
&&\leq3\left(\frac{1}{k}\frac{2k}{n}K^2(T+1)^2\varepsilon^2\cdot 4 n+\frac{1}{k}\frac{k}{n}2\varepsilon^4 \cdot 4n+\frac{1}{k}\frac{2kT}{n}\frac{4}{h}\varepsilon^2nh\right)\\&&\leq 24(K^2(T+1)^2 +1+T)\varepsilon^2.
\end{eqnarray*}
This shows (22) for $Z_{n,i,2}$.
\end{proof}
\section{Simulation Study}
\label{sec:simulation}
We conduct a Monte Carlo simulation study to assess the finite-sample performance of the ``integrated'' estimator and the two testing procedures based on the
supremum-type statistic
and the Cram\'{e}r--von Mises-type statistic,
both defined in Corollary~\ref{corollary:global estimator no trend}. 
We denote the test statistics with $   T_n^{\rm sup} $ and  $
   T_n^{\rm CvM}$ respectively, and take $T=1$.

\subsection{Data generating process}

We simulate a sample $(X_i, Y_i)$, $i = 1, \ldots, n$, satisfying the model in our setup. One key element used in the simulation is the Gumbel copula which leads to the logistic tail copula: for $\theta \in (0,1]$,
\begin{equation}
    \label{eq:sim_logistic}
    R_\theta(u,v) = u + v - \left(u^{1/\theta} + v^{1/\theta}\right)^{\theta},
    \quad u, v \geq 0,
\end{equation}
with tail-dependence coefficient $R_\theta(1,1) = 2 - 2^\theta$. Let $(U_i, V_i)$, $i = 1, \ldots, n$, be an i.i.d.\ sample from the
bivariate logistic distribution with parameter $\theta$. Then each marginal distribution is the standard Fr\'{e}chet(1) distribution, $F_0(x) = \exp(-1/x)$, $x > 0$ and the copula follows the Gumbel copula.

Under the null hypothesis
with constant tail copula, i.e., 
$H_0\colon R(u,v;s) \equiv R_\theta(u,v)$, for given $\theta$ and $s \in [0,1]$, the observations are then transformed by the heteroscedastic scaling
\begin{equation}
    \label{eq:sim_transform}
    X_i = c_X(s_i)\,U_i, \qquad Y_i = c_Y(s_i)\,V_i, \quad s_i = i/n
\end{equation}
Since the copula is invariant under strictly increasing marginal transformations, the tail
copula of $(X_i, Y_i)$ is  $R_\theta$.
Moreover, the scaling by $c_X(s_i)$ and $c_Y(s_i)$ realizes the scedasis functions of each marginal, respectively.
The sample $(X_i, Y_i)$ thus satisfies the conditions in the null hypothesis, with potentially different marginal scedasis functions.

We consider three specifications for the scedasis pair $(c_X, c_Y)$:
\begin{description}
    \item[M1 (Constant).]  $c_X(s) = c_Y(s) = 1$.
    \item[M2 (Linear).]    $c_X(s) = 0.8 + 0.4s$ and $c_Y(s) = 1.5 - s$.
    \item[M3 (Periodic).]  $c_X(s) = 1 + 0.6\cos(2\pi s)$ and
                           $c_Y(s) = 1 + 0.4\sin(2\pi s)$.
\end{description}
All three examples satisfy Condition~\ref{eq:condition smooth of c and R}.
The purpose of M2 and M3 is to examine whether non-trivial marginal heteroscedasticity
affects the behavior, as the theory predicts it should not. We consider $\theta \in \{0.5, 0.9\}$, corresponding to tail-dependence coefficients
$0.59$ (moderate) and $0.13$ (weak).
Combined with M1--M3, this yields six configurations.

In the general case of a varying tail copula, we generate samples of $(X_i, Y_i)$ in which the tail copula
varies smoothly with $s$. Specifically, we design the following function
\begin{equation}
    \label{eq:sim_alternative}
    R(u,v;s) = f(s)\,R_\theta(u,v) ,
\end{equation}
where $R_\theta(u,v)$ is the logistic tail copula as in~\eqref{eq:sim_logistic}, and $f(s)$ is a weight function that mixes this tail copula with tail independence, thus creating a smoothly varying tail dependence structure across $s$.

For the weight function $f(s)$, we use the cubic transition
\begin{equation}
    \label{eq:sim_mixing}
    f(s) = 1 - \frac{0\vee(s - \lambda)^3}{(1 - \lambda)^3}, \quad s \in [0, 1],
\end{equation}
where $\lambda\in (0,1]$ is a parameter. 
For $s \leq \lambda$, $f(s)=1$, i.e. the observations following the logistic tail copula; the dependence then
decreases smoothly, reaching the level of $0$ at $s=1$, i.e. tail independence.
The parameter $\lambda \in [0,1]$ controls the onset: large $\lambda$ implies a late,
short transition (and hence a small departure from $H_0$). 
Note that, for $\lambda = 1$, we define $f\equiv 1$ which corresponds to the null hypothesis $R(u,v;s) = R_\theta(u,v)$. 
Lastly, we remark that  the $f$ function, $R(u,v;s)$ satisfies Condition~\ref{eq:condition partial derivatives of R}.

To simulate observations following such time varying tail copulas, we use a mixture idea. First, draw
$(U_i, V_i)$ from a position-dependent Bernoulli mixture: with probability $f(s_i)$,
$(U_i, V_i)$ follows the bivariate logistic distribution;
with the complementary probability $1 - f(s_i)$, $U_i$ and $V_i$ are drawn independently
from $F_0$. After that, we transform the observations by the same heteroscedastic scaling as in~\eqref{eq:sim_transform}, with the same three scedasis specifications
M1--M3 as defined before.

\subsection{The integrated tail copula estimator}
We examine the finite-sample
behaviour of the integrated tail copula estimator $\widehat{I}_R(u,v;s)$
itself.  We start with the estimator at the single endpoint $(u,v,s)=(1,1,1)$, for which a
closed-form asymptotic distribution is available.

We use the general data generating process, with $\theta = 0.5$.  The marginal
scedasis functions are constant (M1).  We vary
$\lambda \in \{0,\,\tfrac{1}{3},\,\tfrac{2}{3},\,1\}$: recall that a larger $\lambda$
implies a later onset of the transition.
We use $M = 200$ replications, with sample size $n=10000$. The tuning parameters are chosen as $h = 1/10$ and  $k = 200$.

The true value of the average tail dependence coefficient is
\[
 I_R(1) =I_R(1,1;1) = \lambda_\theta\,(3+\lambda)/4, \qquad
  \lambda_\theta = 2 - 2^\theta = 2 - \sqrt{2}.
\]
Theorem \ref{theorem:global estimator for multivariate} yields that the normalised estimator satisfies
\[
  \sqrt{k}\,\bigl(\widehat{I}_R(1) - I_R(1)\bigr)
  \xrightarrow{d} N(0,\sigma^2(1)),
\]
where $\sigma^2(1) = \int_0^1 f(t)\bigl[\lambda_\theta(1-4f(t)r+2f^2(t)r^2)
+ 2f^2(t)r^2\bigr]\,dt$ and $r = \lambda_\theta/2$ is (one of) the
partial derivative(s) of $R_\theta$ at $(1,1)$.

Figure~\ref{fig:estimation} displays, for each $\lambda$, a histogram of the
$M = 200$ normalised estimates overlaid with the standard normal density.
The histograms are centred near zero and closely track the $N(0,1)$ curve
across all four panels.  Results with
$n = 5000$ are qualitatively identical, with slightly higher bias of at most $3\%$.  These results confirm the stated asymptotic normality
of $\widehat{I}_R(1,1;1)$ under non-stationarity.

\begin{figure}[htbp]
\centering
\includegraphics[width=0.7\textwidth]{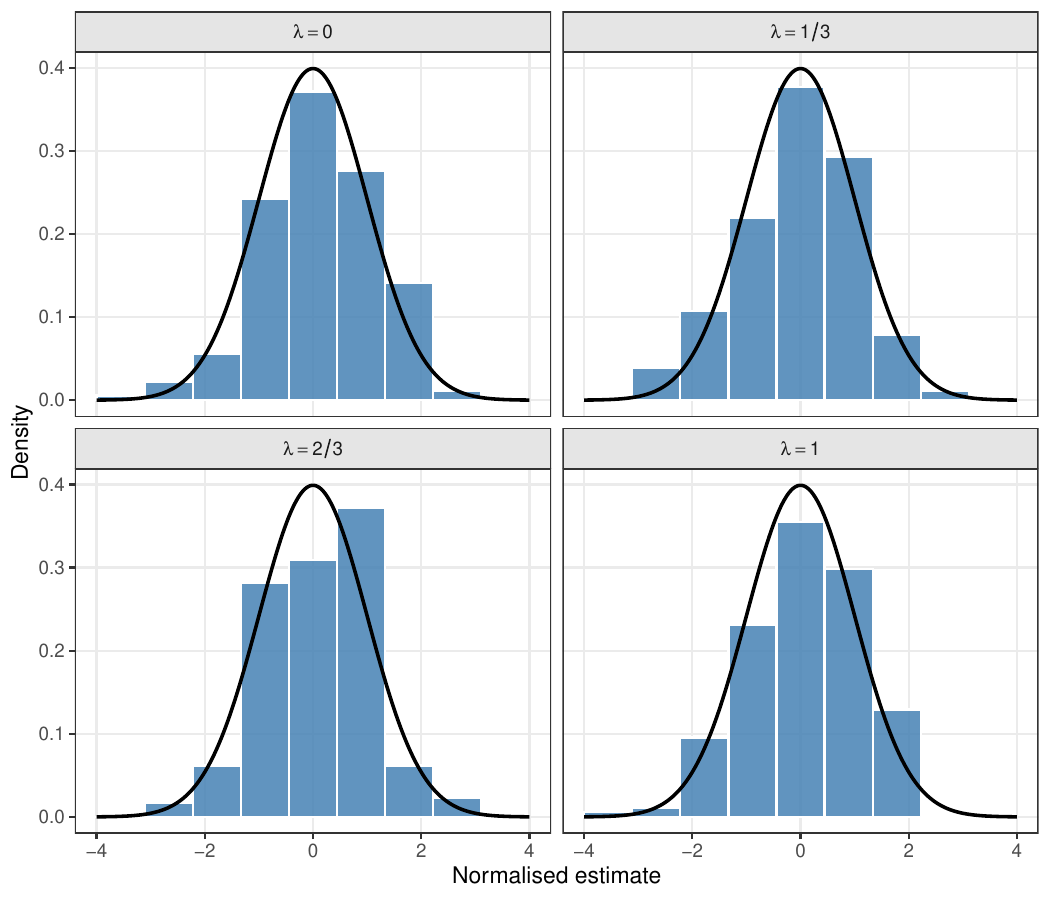}
\caption{Histograms of the normalised estimator
$\sqrt{k}\,(\widehat{I}_R(1)-I_R(1))/\sigma(1)$ across $M=200$
replications, overlaid with the $N(0,1)$ density (solid curve).
Each panel corresponds to one value of $\lambda$.
Parameters: $n=10000$, $k=200$, $h=1/10$, $\theta=0.5$, M1 (constant)
marginals.}
\label{fig:estimation}
\end{figure}

We next examine the estimation of the full curve $I_R(s)=I_R(1,1;s)$ as a function of $s$. We fix $\lambda = 0$, giving true curve $I_R(s) = \lambda_\theta(s - s^4/4)$. We use $n = 5000$, $k = 200$, $h = 1/10$, and all three marginal
specifications M1--M3. Figure~\ref{fig:estimation_curve} shows, for each marginal, the pointwise mean of $\widehat{I}_R(s)$ across $M = 200$ replications together with the empirical $95\%$ band and the true curve. The pointwise mean tracks the truth closely across
$s$, and the results are nearly identical across M1--M3, confirming that
marginal heteroscedasticity does not affect the estimation of $I_R$.

\begin{figure}[htbp]
\centering
\includegraphics[width=\textwidth]{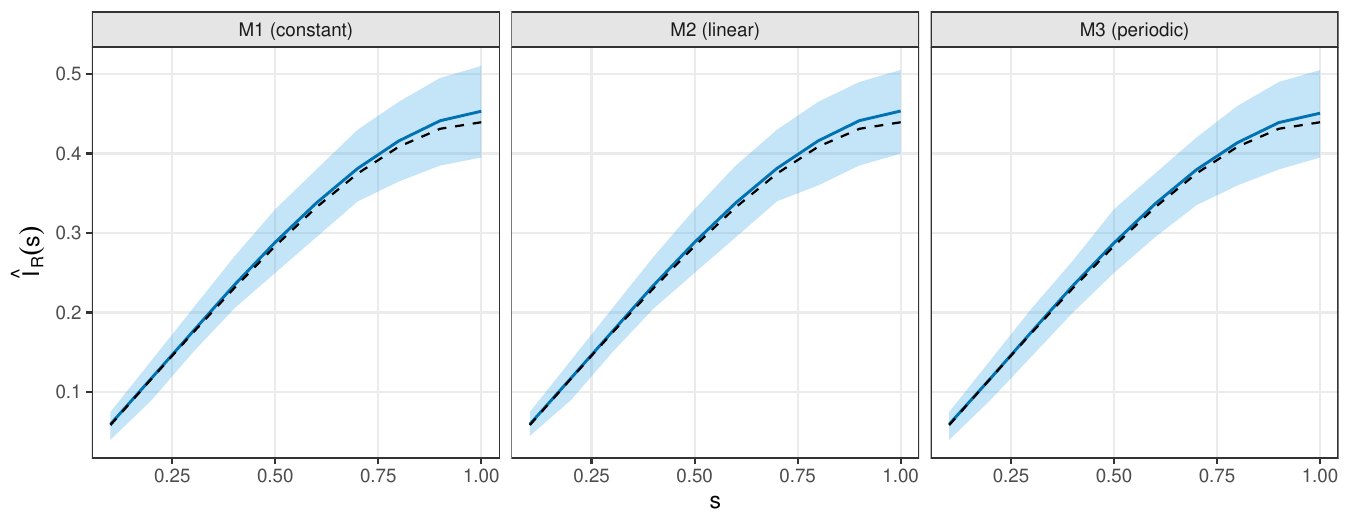}
\caption{Estimated curve $\widehat{I}_R(s)$ (solid) with pointwise empirical
$95\%$ band (shaded) and true curve (dashed) across $M=200$ replications.
Each panel corresponds to one marginal specification.
Parameters: $n=5{,}000$, $k=200$, $h=1/10$, $\theta=0.5$, $\lambda=0$.}
\label{fig:estimation_curve}
\end{figure}
\subsection{Size Analysis}
For the size study, we use the data generating process under the null hypothesis. For taking the supremum or integral, we evaluate the estimators on the discrete grid
$\mathcal{G} = \{0.1, \ldots, 1.0\}^2 \times \{h, 2h, \ldots, 1\}$,
with the integral in $T_n^{\rm CvM}$ replaced by the corresponding Riemann sum.
In both designs below, critical values are obtained by simulating the limiting Gaussian
bridge $\widehat{B}$ conditional on $\widehat{R}$, as in
Corollary~\ref{corollary:global estimator no trend}:
we generate $B = 1000$ independent realizations of $\widehat{B}$, compute the
respective functional of each, and take the empirical 95th percentile as the critical value.
All tests are at nominal level $\alpha = 0.05$.

For each configuration, we simulate $M = 200$ independent
samples of size $n = 5000$.
Both tests are applied to each sample. For the tuning parameters $(k,h)$, we use all possible combinations among
$k \in \{100, 200, 300, 500, 700, 1000\}$ and $h \in \{1/10, 1/15, 1/20\}$,
so that sensitivity to the tuning parameters is assessed.

Figure~\ref{fig:design1_cvm}
displays the empirical rejection
rates as a function of $k$ for the 
CvM test; the results for the supremum test are very similar and hence omitted. 
We show the results for all six configurations in six subfigures. Within each subfigure, the three line types correspond to the three bandwidth choices for $h$.

From this figure we obtain that, both tests achieve the desired size for $k \geq 300$, with rejection rates typically in $[0.02,\,0.06]$. At $k = 100$ there is noticeable underrejection in the left panels, reflecting the limited effective sample size. This underrejection diminishes as $k$ grows, consistent with the asymptotic theory. Among the bandwidth choices, $h = 1/10$ provides the most stable size across $k$, while the other two choices provide virtually indifferent results.
The results are nearly identical across M1--M3, confirming that heteroscedastic
margins do not distort the testing procedure. Variation between $\theta = 0.5$ and $\theta = 0.9$ is also small, demonstrating that the tests maintain size across different levels of tail dependence.


\begin{figure}[h]
\centering
\begin{subfigure}[b]{0.47\textwidth}
    \includegraphics[width=\textwidth]{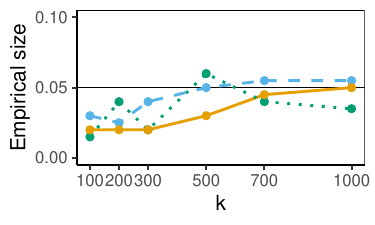}
    \caption{$\theta=0.5$, M1 (constant)}
\end{subfigure}
\hfill
\begin{subfigure}[b]{0.47\textwidth}
    \includegraphics[width=\textwidth]{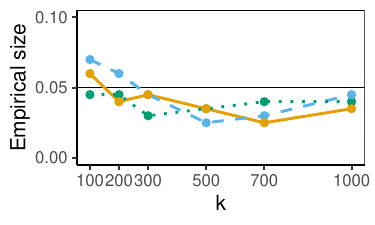}
    \caption{$\theta=0.9$, M1 (constant)}
\end{subfigure}
\\[0.5em]
\begin{subfigure}[b]{0.47\textwidth}
    \includegraphics[width=\textwidth]{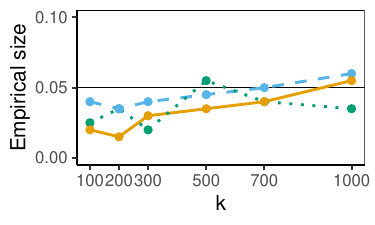}
    \caption{$\theta=0.5$, M2 (linear)}
\end{subfigure}
\hfill
\begin{subfigure}[b]{0.47\textwidth}
    \includegraphics[width=\textwidth]{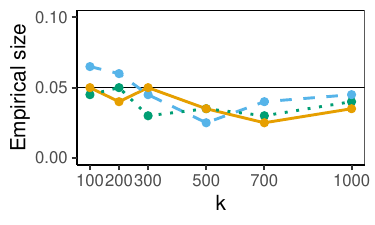}
    \caption{$\theta=0.9$, M2 (linear)}
\end{subfigure}
\\[0.5em]
\begin{subfigure}[b]{0.47\textwidth}
    \includegraphics[width=\textwidth]{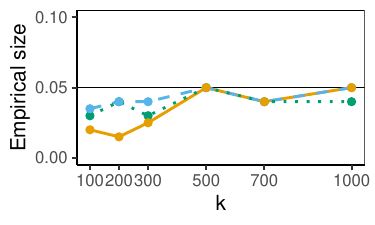}
    \caption{$\theta=0.5$, M3 (periodic)}
\end{subfigure}
\hfill
\begin{subfigure}[b]{0.47\textwidth}
    \includegraphics[width=\textwidth]{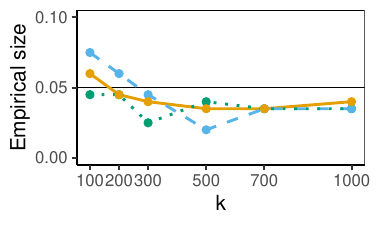}
    \caption{$\theta=0.9$, M3 (periodic)}
\end{subfigure}
\caption{Empirical size of the Cram\'{e}r--von Mises test ($T_n^{\rm CvM}$) as a function
of $k$, at nominal level $\alpha = 0.05$ (horizontal line).
Solid: $h = 1/10$; dashed: $h = 1/15$; dotted: $h = 1/20$.
Based on $M = 200$ replications, $n = 5000$.}
\label{fig:design1_cvm}
\end{figure}

\subsection{Power Analysis}
For the power analysis, we use the data generating process based on the mixture idea with $\theta=0.5$. We simulate  $M = 200$ samples with sample size
$n = 5000$.

Based on the size results, for $n=5000$, we set $k = 1000$ and $h = 1/20$ throughout. The power analysis focuses on the effect of $\lambda$ and the impact of marginal skedasis functions.
We consider $\lambda \in \{0.70, 0.75, 0.80, 0.85, 0.90, 0.95, 1.00\}$. For each $(\lambda, \text{skedasis})$ configuration. 

\begin{figure}[htbp]
\centering
\begin{subfigure}[b]{0.47\textwidth}
    \includegraphics[width=\textwidth]{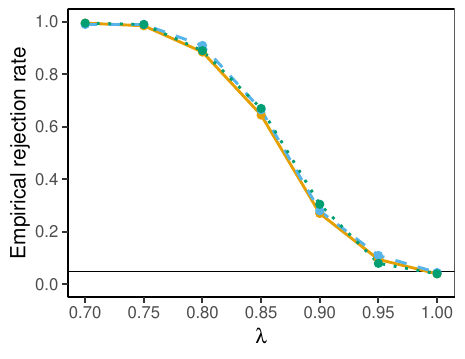}
    \caption{Supremum test ($T_n^{\rm sup}$)}
    \label{fig:design2_sup}
\end{subfigure}
\hfill
\begin{subfigure}[b]{0.47\textwidth}
    \includegraphics[width=\textwidth]{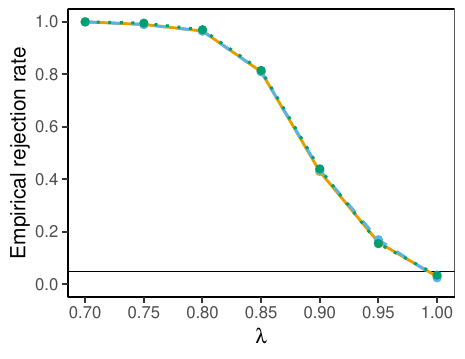}
    \caption{Cram\'{e}r--von Mises test ($T_n^{\rm CvM}$)}
    \label{fig:design2_cvm}
\end{subfigure}
\caption{Empirical power as a function of $\lambda$,
at nominal level $\alpha = 0.05$ (horizontal line).
Fixed $k = 1000$, $h = 1/20$, $\theta = 0.5$.
Solid: M1 (constant); dashed: M2 (linear); dotted: M3 (periodic).
At $\lambda = 1$ the tail copula is constant ($f \equiv 1$), a null scenario.
Based on $M = 200$ replications, $n = 5000$.}
\label{fig:design2}
\end{figure}

Figure~\ref{fig:design2} shows the empirical power as a function of $\lambda$.
Both tests are highly powerful for small $\lambda$: rejection rates exceed $95\%$ for
$\lambda \leq 0.75$ across all marginal specifications.
Power decreases smoothly toward the nominal $5\%$ as $\lambda$ approaches~1, and at
$\lambda = 1$ the rejection rate is indeed close to the size, confirming size control in
this null scenario.

The CvM test exhibits moderately higher power than the supremum test for intermediate
$\lambda$:. For instance, at $\lambda = 0.90$, the CvM test rejects approximately $50\%$ of the time
versus roughly $35\%$ for the supremum test.
As in the size analysis, the power curves are nearly identical across M1--M3, providing
further evidence that marginal heteroscedasticity does not affect the behaviour of the
tests.

\medskip 
{\bf Acknowledgment} We are very grateful to Laurens de Haan for helpful input and discussions at the early stages of this project.

\bibliographystyle{chicago}
\bibliography{depskedasis}

@preamble{ " \providecommand{\noopsort}[1]{} " }

@article{a84,
  title={Probability inequalities for empirical processes and a law of the iterated logarithm},
  author={Alexander, Kenneth S},
  journal={The Annals of Probability},
  pages={1041--1067},
 volume={12},
  number={4},
  year={1984},
  publisher={JSTOR}
}

@article{es21,
  title={Empirical tail copulas for functional data},
  author={Einmahl, J. H. J. and Segers, Johan},
  journal={The Annals of Statistics},
  volume={49},
  number={5},
  pages={2672--2696},
  year={2021},
  publisher={JSTOR}
}

@book{BGST,
  title={Statistics of extremes: theory and applications},
  author={Beirlant, Jan and Goegebeur, Yuri and Segers, Johan and Teugels, Jozef L},
  year={2006},
  publisher={John Wiley \& Sons}
}

@book{dHF,
  title={Extreme value theory: an introduction},
   author={\noopsort{Haan}de Haan, Laurens  and Ferreira, Ana},
  
  year={2006},
  publisher={Springer}
}

@article{dHZ,
  title={Trends in extreme value indices},
  author={\noopsort{Haan}de Haan,   Laurens and Zhou, Chen},
  journal={Journal of the American Statistical Association},
  volume={116},
  number={535},
  pages={1265--1279},
  year={2021},
  publisher={Taylor \& Francis}
}

@article{EdHZ,
  title={Statistics of heteroscedastic extremes},
  author={Einmahl, J. H. J. and  de Haan, Laurens and Zhou, Chen},
  journal={Journal of the Royal Statistical Society Series B: Statistical Methodology},
  volume={78},
  number={1},
  pages={31--51},
  year={2016},
  publisher={Oxford University Press}
}

@article{Drees,
  title={Statistical inference on a changing extreme value dependence structure},
  author={Drees, Holger},
  journal={The Annals of Statistics},
  volume={51},
  number={4},
  pages={1824--1849},
  year={2023},
  publisher={Institute of Mathematical Statistics}
}

@article{EZ,
  title={Tail copula estimation for heteroscedastic extremes},
  author={Einmahl, J. H. J. and Zhou, Chen},
  journal={Econometrics and Statistics},
  year={2026},
  publisher={Elsevier}
}

@article{HH,
  title={Two sample tests for bivariate heteroscedastic
extremes with a changing tail copula},
  author={\noopsort{Haan}Hu, Yifan and Hou,  Yanxi},
  journal={Statistica Sinica},
  year={2028},
  
}

@article{EdHL,
  title={Weighted approximations of tail copula processes with application to testing the bivariate extreme value condition},
  author={Einmahl, J. H. J. and de Haan, Laurens and Li, Deyuan},
  year={2006}
}

@article{EKS,
  title={An M-estimator for tail dependence in arbitrary dimensions},
  author={Einmahl, J. H. J. and Krajina, Andrea and Segers, Johan},
  year={2012}
}
\end{document}